\documentclass[a4paper]{scrarticle}
\usepackage{a4wide}
\usepackage{amscd,amssymb,amsmath,amsthm}
\usepackage{graphicx}
\usepackage[backend=biber, style=ieee,giveninits=true,sorting=anyt]{biblatex}
\usepackage{bbm}
\usepackage{hyperref}
\usepackage{doi}
\hypersetup{
    colorlinks=true,
    linkcolor=blue,
    filecolor=black,      
    urlcolor=blue,
    citecolor=blue,
}
\newcommand{\invisible}[1]{}
\usepackage[section]{placeins}
\usepackage[shortlabels]{enumitem}
\usepackage[lofdepth,lotdepth]{subfig}
\usepackage{tikz}
\usetikzlibrary{arrows,shapes,decorations,automata,backgrounds,petri,bending,trees,decorations.pathreplacing,calc}
\newcommand{\Z}{\mathbb{Z}}
\newcommand{\R}{\mathbb{R}}
\newcommand{\N}{\mathbb{N}}
\newtheorem{thm}{Theorem}[section]
\newtheorem{defn}[thm]{Definition}
\newtheorem{lemma}[thm]{Lemma}
\newtheorem{pro}[thm]{Proposition}
\newtheorem{rk}[thm]{Remark}
\newtheorem{cor}[thm]{Corollary}

\newtheorem{condition}[thm]{Condition}

\newcommand{\footnoteremember}[2]{
	\footnote{#2}
	\newcounter{#1}
	\setcounter{#1}{\value{footnote}}
}
\newcommand{\footnoterecall}[1]{
	\footnotemark[\value{#1}]
}
\newcommand{\sss}{\scriptscriptstyle}
\newcommand{\dphigeqalpha}{d_\phi^{\sss(\geq \alpha)}}
\newcommand{\dphigeqalphadir}{d_\phi^{\sss(-,\geq \alpha)}}
\newcommand{\nn}{\nonumber}
\DeclareSymbolFont{extraup}{U}{zavm}{m}{n}
\DeclareMathSymbol{\varheart}{\mathalpha}{extraup}{86}
\DeclareMathSymbol{\vardiamond}{\mathalpha}{extraup}{87}
\newcommand{\ensymboldefinition}{$\blacktriangleleft$}
\newcommand{\conn}{\longrightarrow}

\newcommand{\eqn}[1]{\begin{equation}#1\end{equation}}
\newcommand{\eqan}[1]{\begin{align}#1\end{align}}

\begin{document}
\title{Power-law hypothesis for PageRank on undirected graphs
}
\author{
 Florian Henning\footnoteremember{TUe}{Eindhoven University of Technology, Department of Mathematics \& Computer Science;\\ 
 f.b.henning@tue.nl, r.w.v.d.hofstad@tue.nl, n.v.litvak@tue.nl}
		\and  Remco van der Hofstad \footnoterecall{TUe}
  \and
	Nelly Litvak\footnoterecall{TUe}
}
\date{\today} 	
\maketitle

\begin{abstract}
Based on observations in the web-graph, the \textit{power-law hypothesis} states that PageRank has a power-law distribution with the same exponent as the in-degree. While this hypothesis has been analytically verified for many random graph models, such as directed configuration models and generalized random graphs, surprisingly it has recently been disproven for the directed preferential attachment model.
In this paper, we prove that in \textit{undirected} networks, the graph-normalized PageRank is always upper bounded by the degree. Furthermore, we  prove that the corresponding (asymptotic) lower bound holds true under reasonable assumptions on the local weak limit, but not in general, and we  provide a counterexample.  
Our result shows that PageRank always has  a \textit{lighter} tail than the degree, which contrasts the case of the directed preferential attachment model, where PageRank has a heavier tail instead.
We end the paper with a discussion, where we extend our results to directed networks with a bounded ratio of in- and out-degrees, and reflect on our methods by contrasting the undirected and directed preferential attachment model. 
\end{abstract}

\textbf{Mathematics Subject Classifications (2020).} 05C80 (primary);\\
60J80, 60B20 (secondary)\\
\textbf{Key words.} PageRank, Scale-free network, heavy tails, maximally assortative, local weak convergence.

\section{Introduction and Results}\label{sec: Introduction}
\subsection{Introduction}
The PageRank algorithm was originally introduced by Page and Brin \cite{BrPa98} to rank web pages by their position in the web-graph, where directed edges are hyperlinks from one web page to another.
An important characteristic of the topology of the web-graph is the \textit{scale-free} property, in that the empirical distribution of both in- and out- degrees is approximately a power law \cite{LaGi98}. In the web-graph, it has also been observed that PageRank has  a power-law distribution with approximately the same exponent as that of the in-degree \cite{PaRaUp06,FoBoFlMe08}; these observations gave rise to the so-called \textit{power-law hypothesis} that on every directed graph a power-law distribution of the in-degrees implies a power-law distribution of PageRank with the same exponent.    
This hypothesis has been recently proven to fail in the preferential attachment model \cite{BaOC22}, one of the  models for the web-graph \cite{BaAl99,AlBaJe00,AdHu2000}; specifically,  the tail of the limiting PageRank distribution in this model is heavier than that of the in-degree \cite{BaOC22}.

For undirected graphs, PageRank is known to correlate highly with the degree. Analytically,  the authors of \cite{AvKaOsPrRa15} show that in undirected semi-sparse expanding graphs, the PageRank-distribution can be asymptotically approximated (w.r.t.\ the total variation distance) by a mixture of the restart distribution and the degree distribution.

\paragraph{\textbf{Main innovation in this paper.}} In this paper we analyze the power-law hypothesis in undirected graphs. Importantly, in the remainder of Section \ref{sec: Introduction}, we prove that for every undirected graph, the graph-normalized PageRank is componentwise bounded from above by the degree vector (Theorem \ref{thm: upperBoundUndirected}). This in particular excludes the asymptotic PageRank distribution from having heavier tails than the degree distribution, in contrast to the directed setting \cite{BaOC22}. 
Additionally, we provide a sufficient condition (Proposition \ref{pro: PageRankLowerBound}) for an asymptotic lower bound on the PageRank distribution. Both bounds together imply that if the graph normalized PageRank has an asymptotic power law, then its exponent is the same as that of the degree distribution, provided that the condition for the lower bound is satisfied. The latter condition, in essence, prevents high-degree vertices from having too many high-degree neighbors, and we expect it to hold quite generally in a large class of random graph models and real-world networks.
In Section~\ref{sec: Applications}, we explicitly check the assumptions for the lower bound for unimodular branching process trees and the P\'olya point tree. Together with the results \cite{GavdHLi20,BeBoCh14,vH24} on local weak convergence of the graph-normalized PageRank distribution, this implies that the power-law hypothesis holds true for, respectively,  the undirected configuration model (Theorem \ref{thm: ConfModell}) and the undirected preferential model (Theorem \ref{thm: PrefAttModell}).

To show that, in contrast to the upper bound, the lower bound is not generally true, in Section \ref{sec: counterexample} we provide an explicit construction (Theorem \ref{thm: Counterxample}) of an evolving sequence of graphs exhibiting a power-law distribution with any prescribed exponent, whose associated PageRank distribution does not exhibit a power-law distribution. 
Finally, the Discussion Section \ref{sec: discussion} deals with extensions (Theorem \ref{thm: UpperBoundBoundedRatio} and Proposition \ref{pro: LowerBoundDirected}) of the results from Section \ref{sec: Introduction} to directed networks assuming that the ratios of in- and out-degrees in the graph are uniformly bounded. 
The discussion closes with a heuristics on the different asymptotic behaviour of PageRank for the undirected- and the directed version of the preferential attachment model, thus explaining how the change in the PageRank power-law arises for the directed setting, while it fails to hold for the undirected setting.
\subsection{Basic notions}
First we start with providing the basic notions which will be employed in the reminder of this paper. For further explanations and introduction to mathematical models for complex networks, we refer the reader to \cite{vdH17,vH24} and \cite{Ne18}.

\paragraph{Definitions.}
We consider a sequence $(G_n)_{n \in \mathbb{N}}$ of finite undirected graphs $G_n=(V_n,E_n)$ without isolated vertices with respective vertex sets $V_n=[n]:=\{1,2,\ldots,n\}$, and edge sets $E_n$. We write $i \sim j$ if two vertices $i,j \in V_n$ are connected by an edge. 
Further, $\boldsymbol{A}^{(\sss G_n)}=(a_{ij}^{(\sss G_n)})_{(i,j) \in [n]} \in \{0,1\}^{n \times n}$ denotes the adjacency matrix of $G_n$, which in our undirected case is symmetric. Finally, let $\boldsymbol{d}^{\sss(G_n)}=(d_i^{\sss(G_n)})_{i \in [n]}$ denote the degree vector corresponding to $G_n$. Without loss of generality, we assume that $d_i^{\sss(G_n)}\geq 1$ for all $i\in [n]$, i.e., there are no isolated vertices. From now on, we will omit the superscript $(G_n)$ in, e.g.,  $d_i^{\sss(G_n)}$ when this notation is clear from the context. 

Throughout the paper, we print vectors and matrices in boldface in contrast to scalars, and define the partial order "$\leq$" on $\mathbb{R}^n$ as  $\boldsymbol{a} \leq \boldsymbol{b}$ if and only if $a_i \leq b_i$ for all $i \in [n]$. As usual, $\Vert \boldsymbol{a} \Vert_p:=\left(\sum_{i=1}^n \vert a_i \vert^p \right)^{1/p}$ denotes the $p$-norm, $p \in (0,\infty)$, on $\R^n$, where we will also abbreviate $\vert \boldsymbol{v} \vert :=\Vert \boldsymbol{v} \Vert_1$.
Furthermore, by $\boldsymbol{a} \odot \boldsymbol{b}$ we denote the elementwise (Hadamard) product of two vectors, i.e., $(\boldsymbol{a} \odot \boldsymbol{b})_i:=a_ib_i, \quad i \in[n]$. 
Finally, $\boldsymbol{1}_n \in \R^n$ denotes the vector of ones and $\mathbbm{I}_n \in \R^{n \times n}$ the $n$-dimensional identity matrix. For notational distinction from the above two symbols, we denote the indicator function on any set $B$ by $\mathbbm{1}_B(\cdot)$.  
We say that a real-valued random variable $X$ is \textbf{stochastically bounded from above} by a real-valued random variable $Y$, denoted by $X \preceq Y$, iff $\mathbb{P}(X \geq x) \leq \mathbb{P}(Y \geq y)$ for every $x \in \R$.
Moreover, we say that (cf.\ the relation $\leq_K$ in \cite[P.\ 27]{St83}) a random vector $\boldsymbol{X} \in \R^n$ is stochastically bounded from above by a random vector $\boldsymbol{Y} \in \R^n$, and write $\boldsymbol{X} \preceq \boldsymbol{Y}$, iff $\mathbb{P}(\boldsymbol{X} \geq \boldsymbol{x}) \leq \mathbb{P}(\boldsymbol{Y} \geq \boldsymbol{x})$ for every $\boldsymbol{x} \in \R^n$. Note that if $\boldsymbol{X}=(X_i)_{i \in [n]}$ and $\boldsymbol{Y}=(Y_i)_{i \in [n]}$ have independent components, then $\boldsymbol{X} \preceq \boldsymbol{Y}$ is equivalent to $X_i \preceq Y_i$ for all $i \in [n]$.
Finally, for functions $f,g\colon \R \rightarrow \R$, we write, for $a \in \R \cup \{-\infty,+\infty \}$, that \textbf{$f = o(g)$ as $x$ tends to $a$} iff $\limsup_{x \rightarrow a}\vert \frac{f(x)}{g(x)} \vert=0$ and \textbf{$f = O(g)$ as $x$ tends to $a$} iff $\limsup_{x \rightarrow a} \vert\frac{f(x)}{g(x)} \vert <\infty$.
\paragraph{The PageRank equation.} Define the stochastic matrix $\boldsymbol{P}=\boldsymbol{P}^{\sss(G_n)}=(p_{ij}^{\sss(G_n)})_{i,j \in [n]} \in [0,1]^{n \times n}$ by setting
\begin{equation}\label{eq: PageRankMatrixUndirect} p_{ij}:=p_{ij}^{\sss(G_n)}=\frac{a_{ij}^{\sss(G_n)}}{d_i^{\sss(G_n)}},
\end{equation}
where $(a_{ij}^{\sss(G_n)})_{i,j\in [n]}$ is the adjacency matrix of $G_n$.
Then the \textbf{graph-normalized PageRank equation with damping factor} $c \in (0,1)$ reads
\begin{equation}\label{eq: PageRank matrix}
\boldsymbol{R}=c\, \boldsymbol{R} \boldsymbol{P}+(1-c)\boldsymbol{1}_n, \qquad \boldsymbol{R} \in [0,n]^n.    
\end{equation}
Note that \eqref{eq: PageRank matrix} can be written equivalently in terms of a fixed-point equation with an irreducible matrix acting on vectors of length $n$, so that the solution $\boldsymbol{R}$ appears as the Perron-Frobenius eigenvector and is hence unique. 
In particular, \eqref{eq: PageRank matrix} is uniquely solved for $\boldsymbol{R}$ by
\begin{equation}\label{eq: PageRankSolutionNormal}
\boldsymbol{R}=(1-c)\boldsymbol{1}_n(\mathbbm{I}_n-c\boldsymbol{P})^{-1}.    
\end{equation}
Componentwise, the above equation equals the Neumann series  
\begin{equation}\label{eq: PageRank Solution}
R_k=(1-c)\sum_{s=0}^\infty c^s\Vert (\boldsymbol{P}^s)_{\cdot,k}\Vert_{1} 
=(1-c)\sum_{s=0}^\infty c^s \sum_{j=1}^n (\boldsymbol{P}^s)_{jk}
, \qquad k\in[n]. 
\end{equation}
\subsection{Results}
\subsubsection{Upper bound on PageRank}
Because of symmetry of the adjacency matrix $\boldsymbol{A}$ of an undirected graph, the next general upper bound holds true:
\begin{thm}[PageRank is at most the degree]\label{thm: upperBoundUndirected}
Given the degree vector $\boldsymbol{d}$ of any  undirected graph $G_n=([n],E_n)$, 
the PageRank vector $\boldsymbol{R}^{\sss(G_n)}$ in \eqref{eq: PageRank matrix} is componentwise upper bounded by the degree, i.e.,
\begin{equation}\label{eq: ub on PageRank}
R_i^{\sss(G_n)} \leq  d_i^{\sss(G_n)} \qquad \text{for all }i \in [n].
\end{equation}
\end{thm}
Note that Theorem \ref{thm: upperBoundUndirected} does not require any further assumptions on the graph or the degree sequence. 
Since the proof of Theorem \ref{thm: upperBoundUndirected} is quite simple, we state it now. It is based on two lemmas. The first, Lemma \ref{pro: transposedPageRank}, states that employing symmetry of the adjacency matrix $\boldsymbol{A}=\boldsymbol{A}^{\sss(G_n)}$ we can go over from Equation \eqref{eq: PageRankSolutionNormal} above to an equivalent expression, where the operator $(\mathbbm{I}_n-c\boldsymbol{P})^{-1}$ functions from the left instead of from the right. Afterwards, Lemma \ref{lem: vonNeumann} provides an upper bound for that new expression.  
\begin{lemma}[Solution PageRank for undirected graphs]
\label{pro: transposedPageRank}
Let $G_n=(V_n,E_n)$ be a finite graph without isolated vertices with $\# V_n=n$. 
Then 
the graph-normalized PageRank equation \eqref{eq: PageRank matrix}
is solved by the undirected PageRank equation given by
\begin{equation}\label{eq: PageRankSymmetric}
\boldsymbol{R}^{\sss(G_n)}=(1-c)\boldsymbol{d}^{\sss(G_n)} \odot (\mathbbm{I}_n-c\boldsymbol{P}^{\sss(G_n)})^{-1}\boldsymbol{Q}^{\sss(G_n)}\boldsymbol{1}_n,
\end{equation}
where $\boldsymbol{Q}=\boldsymbol{Q}^{\sss(G_n)} \in \mathbb{R}^{n \times n}$ is the diagonal degree matrix defined by $Q_{i,i}=1/d_i$.
\end{lemma}
\begin{proof}[Proof of Lemma \ref{pro: transposedPageRank}]
To ease readability, we will omit the superscripts $(G_n)$ in the proof.
The $i$th component, $i \in [n]$, of the PageRank-Equation \eqref{eq: PageRank matrix} reads
\begin{equation}
R_i=c\sum_{j \in [n]}p_{ji}R_j+(1-c).    
\end{equation}
Divide both sides of the equation by the respective degree $d_i$, and insert the definition of $p_{ji}$ which implies that $d_i p_{ij}=a_{ij}=a_{ji}=p_{ji}d_j$, to obtain
\begin{equation}\label{eq: PageRankSymmetry}
\begin{split}
\frac{R_i}{d_i}&=c\sum_{j \in [n]}\frac{a_{ji}}{d_id_j}R_j+\frac{1-c}{d_i}=c\sum_{j \in [n]}p_{ij}\frac{R_j}{d_j}+\frac{1-c}{d_i},
\end{split}
\end{equation}
where the second equation follows from symmetry of the adjacency matrix $\boldsymbol{A}$.

Next define the vector $\boldsymbol{v}$ by $v_i:=R_i/d_i,$ for $i\in[n]$.
In matrix-notation, this change of variables transforms equation
 \eqref{eq: PageRank matrix} to 
\begin{equation}\label{eq: PageRank matrix transposed}
\boldsymbol{v}=c \,\boldsymbol{P} \boldsymbol{v}+(1-c)\boldsymbol{Q}\boldsymbol{1}_n.
\end{equation}
As we have $\Vert c\boldsymbol{P} \Vert=c \Vert \boldsymbol{P} \Vert \leq c <1$, where $\Vert \boldsymbol{P} \Vert:=\sup_{\Vert \boldsymbol{w} \Vert_2=1} \Vert \boldsymbol{P}\boldsymbol{w} \Vert_2$ denotes the operator-norm, we can solve \eqref{eq: PageRank matrix transposed} for $\boldsymbol{v}$ (see, e.g., \cite[Section 5.7]{Al16}) to arrive at 
\begin{equation}\label{eq: UpperBoundAlmostThereWithV}
\boldsymbol{v}=(1-c)(\mathbbm{I}_n-c\boldsymbol{P})^{-1}\boldsymbol{Q}\boldsymbol{1}_n,    
\end{equation}
which, by definition of $\boldsymbol{v}$, is in turn equivalent to the undirected PageRank equation
\begin{equation}\label{eq: UpperBoundAlmostThere}
\boldsymbol{R}=(1-c)\boldsymbol{d} \odot (\mathbbm{I}_n-c\boldsymbol{P})^{-1}\boldsymbol{Q}\boldsymbol{1}_n.     
\end{equation} 
\end{proof}
\begin{lemma}[Bound on undirected PageRank equation]\label{lem: vonNeumann}
The second factor in \eqref{eq: UpperBoundAlmostThere} satisfies $(\mathbbm{I}_n-c\boldsymbol{P})^{-1}\boldsymbol{Q}\boldsymbol{1}_n \leq \boldsymbol{1}_n/(1-c).$
\end{lemma}
\begin{proof}
From the assumption that $d_i \geq 1$ for all $i \in [n],$ we directly have $\boldsymbol{Q}\boldsymbol{1}_n \leq \boldsymbol{1}_n$. 
Then express $(\mathbbm{I}_n-c\boldsymbol{P})^{-1}$ in terms of the Neumann series 
\begin{equation}\label{eq: Von Neumann}
(\mathbbm{I}_n-c\boldsymbol{P})^{-1}=\sum_{s=0}^\infty (c\boldsymbol{P})^s, 
\end{equation} 
and employ the fact that $\boldsymbol{P}$ is a (row-)stochastic matrix that functions from the left to arrive at $(\mathbbm{I}_n-c\boldsymbol{P})^{-1}Q\boldsymbol{1}_n \leq \sum_{s=0}^\infty c^s \boldsymbol{1}_n$, which finishes the proof. 
\end{proof}
\begin{rk}[Relation to \cite{OC21}]
{\normalfont \rmfamily The transformation $v_i:=R_i/d_i$ that we use is similar to the transformation $X_i:=\mathcal{C}\mathcal{R}_i$ employed in \cite{OC21} to describe the limiting PageRank as a  solution to a stochastic fixed-point recursion equation on a suitable weighted branching process tree.}
\hfill \ensymboldefinition
\end{rk}

\subsubsection{Lower bound on PageRank}
For the upper bound in Theorem \ref{thm: upperBoundUndirected}, we have used 
that the degree vector is component-wise lower bounded by $1$ (where we recall that the graph is assumed to have no isolated vertices).
In the context of directed graphs where the degrees satisfy a power-law distribution, however, the family $(d_k^{\sss(G_n)})_{n \in \mathbb{N}, \ k \in [n]}$ is clearly not bounded from above and a reasoning similar to the one in the proof of \eqref{eq: ub on PageRank} will not work to obtain a lower bound on PageRank.  

Nonetheless, in Proposition \ref{pro: PageRankLowerBound}, below we state a \textit{sufficient condition} on rooted random graphs for a lower bound to hold for the respective \textit{PageRank at the root}, which we define in Definition \ref{def: limitingPageRank} below. The result of Proposition \ref{pro: PageRankLowerBound} gets its full strength in combination with  \cite[Theorem 2.1]{GavdHLi20}, which shows that the limiting distribution of the graph-normalized PageRank is determined by the local limit of the graph sequence.
We postpone a brief introduction on local weak convergence and its implications to PageRank to Section \ref{sec: Applications}, and define rooted graphs and PageRank on infinite locally-finite graphs:

\begin{defn}[Rooted graphs]\label{def:rooted-graphs}
{\normalfont \rmfamily By a \textbf{rooted graph} $G_\star$, we mean a pair $(G,\phi)$ such that $G=(V,E)$ is a graph and $\phi \in V$ is a distinguished vertex called the \textbf{root}. We say that $(G,\phi)$ is \textbf{locally finite} when every vertex in $V$ has a finite degree.

Let $\mathcal{G}_\star$ denote the quotient space of the set of all locally-finite connected rooted graphs with respect to the equivalence relation $\cong$ given by root preserving graph-isomorphisms. 
We endow $\mathcal{G}_\star$ with the Borel-$\sigma$ algebra generated by the local metric $d_{loc}$ on $\mathcal{G}_\star$, which we will define later on in Definition \ref{def: local metric} when it is actually needed in the framework of local convergence.}
\hfill\ensymboldefinition
 \end{defn}
  
 \begin{defn}[Root-PageRank on infinite rooted graph]\label{def: limitingPageRank}
 {\normalfont \rmfamily Let $G_\star=(G,\phi)$ be any representative of an element in $\mathcal{G_\star}$ with degree vector $\boldsymbol{d}$ and adjacency matrix $\boldsymbol{A}$.
 Consider the (possibly infinite-dimensional) matrix $\boldsymbol{P}=(p_{ij})_{i,j \in \N}$ which is analogously to \eqref{eq: PageRankMatrixUndirect} defined by $p_{ij}=a_{ij}/d_i$.
The \textbf{PageRank at the root} or \textbf{Root-PageRank} is the random variable defined analogously to \eqref{eq: PageRank Solution} by
\begin{equation}\label{eq: LimitingPageRank}
R_\phi=R_\phi(G)=(1-c)\sum_{s=0}^\infty c^s \sum_{j \in V} (\boldsymbol{P}^s)_{j\phi}. 
\end{equation}
}\hfill\ensymboldefinition
\end{defn}
\begin{rk}[Uniqueness Root-PageRank]
{\normalfont \rmfamily By linearity, the function $R_\phi$ is indeed a function on the quotient space $\mathcal{G}_\star$ and does not depend on the choice of the representative element chosen as $G$.}\hfill\ensymboldefinition
\end{rk}
For the root-PageRank, we have the following asymptotic lower bound. In its statement, we write 
\eqn{
\dphigeqalpha=d_{\phi}^{\sss(G_n,\geq \alpha)}=\#\{j \in V \colon j \sim \phi \text{ and } d_j \geq \alpha\}
}
for the number of neighbors of $d_\phi$ having degree at least $\alpha:$
\begin{pro}[Root-PageRank lower bound]\label{pro: PageRankLowerBound} 
In the set-up of Definition \ref{def: limitingPageRank} assume that there exist $\alpha>0$ and $\varepsilon\in(0,1)$ such that, as $k\rightarrow \infty,$
\begin{equation}\label{eq: LowerBoundMainAssumtion}
\mathbb{P}\left(d_\phi >k, \dphigeqalpha 
\geq (1-\varepsilon)d_\phi\right)=o(\mathbb{P}(d_\phi >k)).
\end{equation}
Then, the solution $R_\phi$ to the root-PageRank-equation \eqref{eq: LimitingPageRank} satisfies 
\begin{equation}\label{cond: o of dphi>K}
\mathbb{P}(R_\phi>k) \geq (1+o(1))\mathbb{P}\Big(d_\phi>\frac{\alpha k}{\varepsilon c(1-c)}\Big) \quad \text{as } k \rightarrow \infty.   
\end{equation}
\end{pro}
Equation \eqref{cond: o of dphi>K} implies that the power-law exponent $\tau_R$ of $R_\phi$, when it exists, is at most as large as the power-law exponent $\tau_d$ of $d_\phi$. Together with the upper bound in Theorem \ref{thm: upperBoundUndirected}, this lower bound implies that if $d_\phi$ follows a power-law with exponent $\tau_d$, then so does $R_\phi$. Thus, the power-law hypothesis holds.

In words, \eqref{eq: LowerBoundMainAssumtion} means that the probability that the degree of the root is large and a positive proportion of neighbors of the root $\phi$ has degree at least $\alpha$ for some fixed, and possibly large, $\alpha$, vanishes compared to the probability that the degree of the root is large. For \eqref{eq: LowerBoundMainAssumtion} to be false, {\em most} of the neighbors of high-degree vertices should have a high degree themselves as well. This means that the graph is highly assortative in a very strong sense. For most random graph models, such a property is not true, and neighbors of high-degree vertices have degrees that are bounded by $\alpha$ with high probability for large $\alpha$. Thus, condition \eqref{eq: LowerBoundMainAssumtion} is quite reasonable.

\begin{proof}[Proof of Proposition \ref{pro: PageRankLowerBound}]
By restricting the outer sum in \eqref{eq: LimitingPageRank} to the first two terms, we obtain
\begin{equation}
\mathbb{P}(R_\phi>k) \geq \mathbb{P}\Big(1-c+c(1-c)\sum_{j \sim \phi}\frac{a_{j\phi}}{d_j}>k\Big) \geq \mathbb{P}\Big(c(1-c)\sum_{j \sim \phi}\frac{1}{d_j}>k\Big). 
\end{equation}
Now, a case distinction whether the number of neighbors of the root $\phi$ whose degree is bounded from above by $\alpha$, is bounded from below by $\varepsilon d_\phi$ or not gives

\begin{align}\label{eq: LowerBoundSplit}
\mathbb{P}\Big(c(1-c)\sum_{j \sim \phi}\frac{1}{d_j}>k\Big)
&\geq \mathbb{P}\Big(c(1-c)\sum_{j \sim \phi}\frac{1}{d_j}>k,\#\{j\colon  j \sim \phi \text{ and } d_j < \alpha\} \geq \varepsilon d_\phi\Big)\\
&\geq \mathbb{P}\Big(d_\phi>\frac{\alpha}{\varepsilon c(1-c)}k,\#\{j \colon j \sim \phi \text{ and } d_j < \alpha\} \geq \varepsilon d_\phi\Big)\nn\cr
&\geq \mathbb{P}\Big(d_\phi>\frac{\alpha}{\varepsilon c(1-c)}k\Big)-\mathbb{P}\Big(d_\phi>\frac{\alpha}{\varepsilon c(1-c)}k, \dphigeqalpha  \geq (1-\varepsilon) d_\phi\Big).\nn 
\end{align}
By Condition \eqref{eq: LowerBoundMainAssumtion},
\begin{equation}
\begin{split}
\frac{\mathbb{P}(d_\phi>l, \dphigeqalpha \geq (1-\varepsilon) d_\phi)}{\mathbb{P}(d_\phi>l)} \stackrel{l \rightarrow \infty}{\longrightarrow} 0,
\end{split}
\end{equation}
which in combination with \eqref{eq: LowerBoundSplit} gives \eqref{cond: o of dphi>K}.
\end{proof}
We expect Condition \eqref{eq: LowerBoundMainAssumtion} to hold very broadly. For its practical application to branching process trees in Section \ref{sec: Applications}, we provide the following Corollary \ref{cor: lowerbound}, that features two different conditions which imply \eqref{eq: LowerBoundMainAssumtion} and are easy to verify in specific (tree) settings:
\begin{cor}[Assumoptions implying \eqref{eq: LowerBoundMainAssumtion}]
\label{cor: lowerbound}
Conditionally on $d_\phi=k$, let $\boldsymbol{N}_{\phi,k}\subset \{j \in V \colon j \sim \phi\}$ be a subset of the neighbors of $\phi$ such that the sequence $(\boldsymbol{N}_{\phi,k})_{k \in \N}$ is almost surely non-decreasing with $k$ and, almost surely as $k \rightarrow \infty$,  
\[\bar{n}_{\phi,k} :=\# \{j \in V \colon j \sim \phi\} \setminus \boldsymbol{N}_{\phi,k}=o(k)\]
\begin{enumerate}[a)]
    \item Fix $k,l$ such that $\mathbb{P}(d_\phi=k, \bar{n}_{\phi,k}=l)>0$. Assume that, conditionally on $d_\phi=k$ and $\bar{n}_{\phi,k}=l$, the family $(d_j)_{j \in \boldsymbol{N}_{\phi,k}}$ is stochastically bounded from above by an independent, but not necessarily identically distributed, family $(\tilde{d_j})_{j \in \boldsymbol{N}_{\phi,k}}$ such that there is a uniform upper bound $C < \infty$ for $(\mathbb{E}[\tilde{d}_j])_{j \in \boldsymbol{N}_{\phi,k}}$ which depends neither on $k$ nor on $l$. Then Condition \eqref{eq: LowerBoundMainAssumtion} holds true for  $\alpha=2C$ and any $\varepsilon<\tfrac{1}{2}$. 
    
    \item Fix $k,l$ such that $\mathbb{P}(d_\phi=k, \bar{n}_{\phi,k}=l)>0$. Assume that, conditionally on $d_\phi=k$ and $\bar{n}_{\phi,k}=l$, the family $(d_j)_{j \in \boldsymbol{N}_{\phi,k}}=l$ is stochastically bounded from above by an i.i.d.\ family $(\tilde{d_j})_{j \in \boldsymbol{N}_{\phi,k}}$ of $\N$-valued random variables, whose distribution depends neither on $k$ nor on $l$. Then Condition \eqref{eq: LowerBoundMainAssumtion} holds true for  $\alpha=\bar{F}_{\tilde{d}_1}^{\leftarrow}(\frac{1}{2}):=\min\{t \in \N_0 \mid \mathbb{P}(\tilde{d}_1 \geq t) \leq \frac{1}{2}\}$ and any $\varepsilon<\tfrac{1}{2}$.
\end{enumerate}
In both cases a) and b), for any $\beta>2\alpha/[c(1-c)]$, 
\begin{equation*}
\mathbb{P}(R_\phi>k) \geq (1+o(1))\mathbb{P}(d_\phi>\beta k) \quad \text{as }k \rightarrow \infty.  
\end{equation*}
\end{cor}
\begin{proof}[Proof of Corollary \ref{cor: lowerbound}]
Let Assumption a) or b) be fulfilled and choose $\alpha$ accordingly. For any $k \in \N$, consider the family $(Y_j)_{j \in \boldsymbol{N}_{\phi,k}}$ of independent random variables defined by 
\begin{equation}
Y_j:= \mathbbm{1}_{\{\tilde{d}_j \geq \alpha\}}.
\end{equation}
Then
\[\mathbb{E}[Y_j]=\mathbb{P}(Y_j =1) \leq \frac{1}{2},
\] 
which in Case a) follows from the Markov inequality and in Case b) from the definition of the function $\bar{F}_{\tilde{d}_1}^{\leftarrow}$.
In particular, $(Y_j)_{j \in \boldsymbol{N}_{\phi,k}}$ is stochastically bounded from above by an i.i.d.\ family $(Z_j)_{j \geq 1}$ of Bernoulli$(\tfrac{1}{2})$-random variables.
The remainder of the proof is the same for both cases.

For all $k \in \N$ with $\mathbb{P}(d_\phi=k)>0$, the assumption on $(\tilde{d}_j)_{j \in \boldsymbol{N}_{\phi,k}}$ guarantees that 
\eqan{\label{eq: ConsequenceOfMarokov}
\mathbb{P}(\dphigeqalpha \geq (1-\varepsilon) d_\phi \mid d_\phi=k)
&\leq \mathbb{P}(\#\{j \in \dphigeqalpha \cap \boldsymbol{N}_{\phi,k}\} \geq (1-\varepsilon) d_\phi-\bar{n}_{\phi,k} \mid d_\phi=k) \cr
&\leq \mathbb{P}(\#\{j \in \boldsymbol{N}_{\phi,k}\text{ and } \tilde{d}_j \geq \alpha\} \geq (1-\varepsilon) d_\phi -\bar{n}_{\phi,k} \mid d_\phi=k) \cr
&\leq \mathbb{P}\Big(\frac{1}{k}\sum_{j \in \boldsymbol{N}_{\phi,k}}Y_j \geq 1-\varepsilon -\frac{\bar{n}_{\phi,k}}{k}\mid d_\phi=k\Big)\cr
&\leq \mathbb{P}\Big(\frac{1}{k}\sum_{j \in \boldsymbol{N}_{\phi,k}}Z_j \geq 1-\varepsilon-\frac{\bar{n}_{\phi,k}}{k} \mid d_\phi=k\Big).
}
Here, the second inequality follows from total probability, writing 
\eqan{
\label{eq: TotalExpUB}
&\mathbb{P}(\#\{j \in\boldsymbol{N}_{\phi,k}\colon d_j\geq \alpha\} \geq (1-\varepsilon) d_\phi-\bar{n}_{\phi,k} \mid d_\phi=k)\\
&=\sum_{l \in \N}\mathbb{P}(\#\{j \in \boldsymbol{N}_{\phi,k}\text{ and } d_j \geq \alpha\} \geq (1-\varepsilon) d_\phi-\bar{n}_{\phi,k} \mid d_\phi=k, \bar{n}_{\phi,k}=l)\mathbb{P}(\bar{n}_{\phi,k}=l \mid d_\phi=k)\cr
& \leq \sum_{l \in \N}\mathbb{P}(\#\{j \in \boldsymbol{N}_{\phi,k}\text{ and } \tilde{d}_j \geq \alpha\} \geq (1-\varepsilon) d_\phi-\bar{n}_{\phi,k} \mid d_\phi=k, \bar{n}_{\phi,k}=l)\mathbb{P}(\bar{n}_{\phi,k}=l \mid d_\phi=k),\nn
}
as, conditionally on $d_\phi=k$ and $\bar{n}_{\phi,k}=l$, the $\R^l$-valued random vector $(d_j)_{j \in \boldsymbol{N}_{\phi,k}}$ is stochastically bounded from above by $(\tilde{d}_j)_{j \in \boldsymbol{N}_{\phi,k}}$ by assumption.
Now choose $0<\varepsilon<\varepsilon^*<\tfrac{1}{2}$, and let $\hat{K}_1$ be large enough such that $1-\varepsilon-\frac{\bar{n}_{\phi,k}}{k} \geq 1-\varepsilon^* > \frac{1}{2}=\mathbb{E}[Z_j]$ almost surely, for all $k \geq \hat{K}_1$.

For $k \geq \hat{K}_1$, a corollary of Cram\'er's Theorem (see \cite[Remark c below Theorem 2.2.3]{DeZo10}) gives
\begin{equation}\label{eq: CramerProp}
\begin{split}
\mathbb{P}\Big(\frac{1}{k}\sum_{j \in \boldsymbol{N}_{\phi,k}}Z_j \geq 1-\varepsilon-\frac{\bar{n}_{\phi,k}}{k} \mid d_\phi=k\Big) &\leq \mathbb{P}\Big(\frac{1}{\# \boldsymbol{N}_{\phi,k}}\sum_{j \in \boldsymbol{N}_{\phi,k}}Z_j \geq 1-\varepsilon^* \mid d_\phi=k\Big)\cr
&\leq 2\exp(- \# \boldsymbol{N}_{\phi,k} \Lambda^*(1-\varepsilon^*)) 
\end{split}
\end{equation}
where $\Lambda^*(1-\varepsilon^*)=(1-\varepsilon^*)\log(2(1-\varepsilon^*))+\varepsilon^* \log(2\varepsilon^*)>0$ is the minimum of the Legendre transform of the moment generating function of $Z_1$ on the interval $[1-\varepsilon^*,1]$.
As the r.h.s.\ in \eqref{eq: CramerProp} decreases with increasing $k$, we conclude that, for any $l \in \mathbb{N}_0$,
\begin{equation}
\mathbb{P}\Big(\frac{1}{k}\sum_{j \in \boldsymbol{N}_{\phi,k+l} }Z_j \geq 1-\varepsilon^* \mid d_\phi=k+l\Big) \leq 2\exp(- \# \boldsymbol{N}_{\phi,k+l} \Lambda^*(1-\varepsilon^*)),    
\end{equation}
which, in combination with the law of total probability, gives
\begin{equation}
\mathbb{P}\Big(\frac{1}{k}\sum_{j \in \boldsymbol{N}_{\phi,k}}Z_j \geq 1-\varepsilon \mid d_\phi \geq k\Big) \stackrel{k \rightarrow \infty}{\rightarrow} 0 \quad \text{exponentially fast.}     
\end{equation}
In combination with \eqref{eq: ConsequenceOfMarokov}, this shows that Condition \eqref{eq: LowerBoundMainAssumtion} holds for $\alpha$ and $\varepsilon>\tfrac{1}{2}$, which in view of Proposition \ref{pro: PageRankLowerBound} concludes the proof. 
\end{proof}
\section{Applications}\label{sec: Applications}
In this section we briefly explain the notion of local weak convergence of random graphs and its implications for the asymptotic of the associated root-PageRank. This part is mostly based on \cite{GavdHLi20}.
Additionally, we define the notion of \textit{power-law decay}.
Afterwards we prove that the assumptions of Corollary \ref{cor: lowerbound} hold true for two particular instances of undirected random graphs/ branching processes. These are unimodular branching process trees and the P\'olya point tree.
This then implies correctness of the power-law hypothesis for every evolving sequence of undirected graphs which converge locally weakly to one these two random graphs, in particular the undirected configuration model and the undirected preferential attachment model.

\subsection{Local weak convergence and the limiting PageRank distribution} Local weak convergence was introduced in \cite{AldSte04, BeSc01} to make the notion that finite graphs can look like infinite graphs from the perspective of a random vertex precise. Local weak limits of random graphs are random elements in the set $\mathcal{G}_\star$ of locally-finite connected rooted graphs up to equivalence by isomorphism (recall Definition \ref{def: limitingPageRank} above). 
For the concept of local weak convergence, the following two-step procedure is of relevance:
First we draw a sequence of finite undirected deterministic or random graphs $G_n=(V_n,E_n)$ with $\# V_n=n$. Second, for each $n \in \N$ we choose a vertex $\phi_n \in V_n$ uniformly at random to obtain a rooted graph $(G_n,\phi_n)$.
Local weak convergence of the sequence $(G_n)_{n \in \N}$ against a (possibly random) $G_\star \in \mathcal{G}_\star$ then means that the double-expectation (w.r.t.\ both stages of randomness) of every bounded continuous (w.r.t. the local metric of Definition \ref{def: local metric} below) function $f\colon \mathcal{G}_\star \rightarrow \R$ converges to the expectation of $f$ w.r.t.\ to the prescribed limiting distribution on $\mathcal{G}_\star$.
To make this precise, for any finite graph $G=(V,E)$, define the (random) probability measure (cf.\ the measure $\mu_H$ in \cite[Page 2]{BeSc01}) $\mathcal{P}(G)$ on $\mathcal{G}_\star$
by setting
\begin{equation}\label{eq: RootVertexDistribution}
\mathcal{P}(G):= \frac{1}{\# V}\sum_{i \in V}\delta_{(G,i)},
\end{equation}
i.e., $\mathcal{P}(G)$ describes the law of the graph $G$ observed from a uniformly chosen root among its vertices. 

Now we describe the relevant topology on $\mathcal{G}_\star$. For a rooted connected graph $(G,\phi)$ let $B^{(\sss G)}_s(\phi)$ denote the subgraph of $(G,\phi)$ of vertices at distance at most $s$ from the root $\phi$.      
\begin{defn}[Definition 3.3 in \cite{GavdHLi20}]\label{def: local metric}
{\normalfont \rmfamily The function $d_{loc}$ on $\mathcal{G}_\star \times \mathcal{G}_\star$, defined by \[d_{loc}\left((G,\phi),(G^\prime,\phi^\prime) \right):=\frac{1}{1+\inf_{s \geq 1}\{B^{(\sss G)}_s(\phi) \ncong B^{(\sss G^\prime)}_s(\phi^\prime) \}}\]
is called the \textbf{local metric} on $\mathcal{G}_\star$.}\hfill\ensymboldefinition
\end{defn}
Note that we can adapt the metric $d_{loc}$ to the set-up of marked and, in particular, directed graphs (see \cite[Definition 2.10]{vH24}).
With these notions in mind, the formal definition of local weak convergence reads as follows:
\begin{defn}[Local weak convergence]\label{def: lwconv}
{\normalfont \rmfamily 
A deterministic or random sequence $(G_n)_{n \in \mathbb{N}}$ of random graphs is said to converge \textbf{locally} weakly to a (possibly random) rooted graph $(G,\phi) \in \mathcal{G}_\star$ with law $\mu \in \mathcal{M}_1(\mathcal{G}_\star,\mathcal{B}(\mathcal{G}_\star))$ iff for any bounded function $f\colon  \mathcal{G}_\star \rightarrow \R$ which is continuous with respect to the topology induced by $d_{loc}$,
\begin{equation}\label{eq: LocalWeakConvDef}
\mathbb{E}[\mathbb{E}_{\mathcal{P}(G_n)}[f(G_n,\phi_n)]] \stackrel{n \rightarrow \infty}{\rightarrow} \mathbb{E}_{\mu}[f(G,\phi)]. 
\end{equation}
Here, the outer expectation is taken with respect to the distribution of $G_n$ and the argument $(G_n,\phi_n)$ of $f$ is restricted to the connected component of $\phi_n$ in $G_n$, which is an element in $\mathcal{G}_\star$}.\hfill\ensymboldefinition
\end{defn}
\begin{rk}[Uniqueness of local weak limit]
{\normalfont \rmfamily The space $(\mathcal{G}_\star,d_{loc})$ is Polish (e.g., \cite[Thm. A.7]{vH24}), hence the (local) weak limit is unique (e.g., \cite[Rem. 13.13]{Kl20}).}\hfill\ensymboldefinition
\end{rk}
\begin{rk}[Convergence of typical neighborhoods]\label{rk: LWCriterion}
{\normalfont \rmfamily To prove local weak convergence in the sense of Definition \ref{def: lwconv} above it suffices to verify \eqref{eq: LocalWeakConvDef} for all test functions $f^H_s$ of the form $f^H_s(G,\phi):=\mathbbm{1}_{\{B_s^{(\sss G)}(\phi) \cong H_\star\}}$, where $s \geq 1$ and $H_\star \in \mathcal{G}_\star$. (e.g.,  \cite[Thm. 2.15]{vH24})}\hfill\ensymboldefinition
\end{rk}
Next we relate the above definition to the root-PageRank $R_\phi$. By truncating the outer sum in \eqref{eq: LimitingPageRank} at some level $N$ (for reasons of continuity) and carefully controlling errors while letting $N$ and $n$ tend to infinity \cite{GavdHLi20} proves that \eqref{eq: LocalWeakConvDef} actually holds true for test functions of the desired form $f_r(G,\phi):=\mathbbm{1}_{\{R_{\phi}(G)>r\}}$. Here $r>0$ is any continuity point of the cumulative distribution function of the root-PageRank of the limiting graph:
\begin{thm}[First part of Theorem 2.1 in \cite{GavdHLi20}]\label{thm: localConvPageRank}
Let $(G_n)_{n \in \N}$ be a sequence of deterministic or random graphs and let $\phi_n \in V_n$ be chosen uniformly at random.
If $(G_n)_{n \in \N}$ converges locally weakly to $(G,\phi) \in \mathcal{G}_\star$, then $\mathbb{E}[R_\phi(G)] \leq 1$ and
\begin{equation}
\begin{split}    
R_{\phi_n}^{\sss (G_n)} &\stackrel{n \rightarrow \infty}{\rightarrow} R_\phi(G) \quad \text{in distribution, i.e.,} \cr 
\mathbb{P}(R_{\phi_n}^{\sss (G_n)} >r) &\stackrel{n \rightarrow \infty}{\rightarrow} \mathbb{P}(R_\phi(G)>r) \text{ at any continuity point $r>0$ of the right-hand side.}
\end{split}
\end{equation}
\end{thm}
\subsection{Power-law decay of degree distribution}
Let us recall the following analytic definitions:
\begin{defn}[Power-law distributions]\label{def: Powerlaw}
\begin{enumerate}
{\normalfont \rmfamily \item A function $\mathcal{L}: (0,\infty )\rightarrow (0,\infty)$ is called \textbf{slowly varying at infinity} iff for any $a>0,$
\[\frac{\mathcal{L}(ax)}{\mathcal{L}(x)} \stackrel{x \rightarrow \infty}{\rightarrow} 1.\]
\item A function $f\colon  (0,\infty) \rightarrow (0,\infty)$ is called \textbf{regularly varying at infinity with tail index $\alpha>0$ and tail function $\mathcal{L}$} iff $\mathcal{L}$ is slowly varying at infinity and  
\[f(x)=\mathcal{L}(x) x^{-\alpha} \text{ for all } x \in (0,\infty). \]
\item A random variable $X$ is said to have a \textbf{power-law distribution with exponent $\tau>1$ and tail function $\mathcal{L}$} iff the complementary  distribution function 
\[\bar{F}_X\colon (0,\infty)\rightarrow [0,1] \quad \text{given by}\quad  \bar{F}_X(x)=\mathbb{P}(X>x)\ \] is regularly varying at infinity with tail index $\tau-1$ and tail function $\mathcal{L}$.
\item A random variable $X$ satisfies \textbf{power-law bounds with exponent $\tau>1$ and tail function $\mathcal{L}$} iff $\mathcal{L}$ is slowly varying at infinity and there exist constants $0<\underline{a} \leq \bar{a}$ such that
\[\underline{a}\mathcal{L}(k)k^{-(\tau-1)} \leq \bar{F}_X(k) \leq \bar{a}\mathcal{L}(k)k^{-(\tau-1)}.\]

}\hfill\ensymboldefinition
\end{enumerate}
\end{defn}

\subsection{Application to branching-process trees and the configuration model}
The unimodular branching-process tree with integrable root-degree distribution $(p_k)_{k \in \N}$ is characterized by the property that every vertex apart from the root has an offspring distribution described by the respective size-biased distribution $p_k^*=(k+1)p_{k+1}/\mathbb{E}[d_\phi]$ (see \cite[Definition 1.26]{vH24} or \cite{AlLy07} for a relation of that notion to unimodular groups).
As the family of degrees of each of the children of the root is in particular i.i.d., the unimodular branching process tree clearly meets the assumptions of Corollary \ref{cor: lowerbound}.
Consequently, if the root-degree distribution $d_\phi$ follows a power law, then the limiting root-PageRank associated with any sequence of undirected random or deterministic graphs which converges locally weakly to that unimodular branching process tree satisfies power-law bounds with the same exponent and tail-function (see Theorem \ref{thm: ConfModell} below).

As a particular application, we consider the \textbf{undirected configuration model}.
The undirected configuration model $\mathrm{CM}_n(\boldsymbol{d})$ is a model for a random graph of size $n$ which has a prescribed degree vector $\boldsymbol{d}\in \N_0^n$ as parameter. Each vertex $i\in[n]$ is assigned $d_i$ \textit{half-edges}. Now, a random graph is sampled by choosing any permutation 
$\varphi\colon  \{1,2,\ldots, \vert \boldsymbol{d}\vert\} \rightarrow \{1,2,\ldots, \vert \boldsymbol{d} \vert\}$ uniformly at random and afterwards connecting the $j$th half-edge to the $\varphi(j)$th half
-edge, where $j$ runs from $1$ to $\vert \boldsymbol{d} \vert$.
\begin{condition}[Degree regularity condition]
\label{cond-degrees}Let $D_n=d_{\phi_n}$ denote the degree of a uniform vertex in $[n]$. We impose the following two assumptions on the degree sequence:
\begin{enumerate}
    \item $D_n \stackrel{n \rightarrow \infty}{\longrightarrow} D$ \text{ in distribution}; \textbf{ and }
    \item $\mathbb{E}[D_n] \stackrel{n \rightarrow \infty}{\longrightarrow} \mathbb{E}[D]$.
\end{enumerate}
\end{condition}
For the directed configuration model, the power-law hypothesis is proven to hold true \cite{ChLiOl17}, and the following theorem states that this is also the case for the undirected configuration model $\mathrm{CM}_n(\boldsymbol{d})$ associated with the sequence $(\boldsymbol{d}_n)_{n \in \N}$:
\begin{thm}[Power-law hypothesis for undirected configuration model]\label{thm: ConfModell}
Consider the undirected configuration model $G_n=\mathrm{CM}_n(\boldsymbol{d}^{\sss(G_n)})$ where the degree distribution satisfies Condition \ref{cond-degrees}. Let $c \in [0,1]$ be a constant.
Then the PageRank vector $\boldsymbol{R}^{\sss(G_n)}$ satisfies that, for all $n,k$,
\begin{equation}\label{eq: UBConfMod}
\mathbb{P}(R_{\phi_n}^{\sss(G_n)}>k) \leq \mathbb{P}(D_n> k),
\end{equation}
while further, for any $\beta>\frac{4\mathbb{E}[D]}{c(1-c)}$,
\begin{equation}\label{eq: LBConfMod}
\liminf_{k\rightarrow \infty}
\liminf_{n \rightarrow \infty} \frac{\mathbb{P}(R_{\phi_n}^{\sss(G_n)}>k)}{\mathbb{P}(D_n>\beta k)} \geq 1.
\end{equation}
In particular, if the limiting degree distribution at the root $D$ has a power-law distribution with any exponent $\tau>1$ and tail function $\mathcal{L}$, then there is some $0< \underline{a} \leq 1$ such that the limiting root-PageRank $R_\phi$ satisfies the power-law bounds 
\begin{equation}
\underline{a}\mathcal{L}(k)k^{-(\tau-1)} \leq \mathbb{P}(R_\phi>k) \leq \mathcal{L}(k)k^{-(\tau-1)}  \quad \text{ for every } k \in (0,\infty).    \end{equation}
\end{thm}

\begin{proof}
The upper bound \eqref{eq: UBConfMod} follows directly from Theorem \ref{thm: upperBoundUndirected}.
By \cite[Theorem 4.1]{vH24}, the undirected configuration model $\mathrm{CM}_n(\boldsymbol{d}_n)$ converges locally in probability, and hence locally weakly, to the unimodular branching process tree. By Corollary \ref{cor: lowerbound} a), the limiting root-PageRank satisfies
  \begin{equation}\label{eq: ConfModelLowBound}
  \mathbb{P}(R_\phi>k) \geq (1+o(1)) \mathbb{P}(D>\beta k) \quad \text{as }k \rightarrow \infty.   
  \end{equation}
On the other hand, by Theorem \ref{thm: localConvPageRank} and the Portmanteau lemma (e.g.\ \cite[Thm. 13.16 v]{Kl20}) applied to the respective open interval $(x,\infty)$, for every $x \in (0,\infty)$,
\begin{equation}\label{eq: PageRankfraction}
\liminf_{n \rightarrow \infty}\frac{\mathbb{P}(R_{\phi_n}^{\sss(G_n)}>x)}{\mathbb{P}(R_{\phi}>x)} \geq 1.    
\end{equation}
Hence, combining \eqref{eq: ConfModelLowBound} and \eqref{eq: PageRankfraction} gives
\begin{equation*}
\begin{split}
&\liminf_{k \rightarrow \infty}\liminf_{n \rightarrow \infty}\frac{\mathbb{P}(R_{\phi_n}^{\sss(G_n)}>k)}{\mathbb{P}(D_n>\beta k)} \cr
&=\liminf_{k \rightarrow \infty} \left(\frac{\mathbb{P}(R_\phi>k)}{\mathbb{P}(D>\beta k)} \liminf_{n \rightarrow \infty}\frac{\mathbb{P}(R_{\phi_n}^{\sss(G_n)}>k)}{\mathbb{P}(R_{\phi}>k)}\frac{\mathbb{P}(D>\beta k)}{\mathbb{P}(D_n>\beta k)} \right)\geq 1.  \end{split}  
\end{equation*}
Here, convergence in distribution in combination with the fact that both $D_n$ and $D$ are integer-valued guarantees that at any fixed $k$, the fraction$\frac{\mathbb{P}(D>\beta k)}{\mathbb{P}(D_n>\beta k)}$ converges to $1$ as $n$ tends to infinity. This proves \eqref{eq: LBConfMod}.
To prove the power-law bounds, first write
\begin{equation}
\frac{\mathbb{P}(D>k)}{\mathbb{P}(R_\phi>k)}
=\liminf_{n \rightarrow \infty} \frac{\mathbb{P}(D> k)}{\mathbb{P}(D_n>k)}\frac{\mathbb{P}(D_n> k)} {\mathbb{P}(R_{\phi_n}^{\sss(G_n)}>k)}\frac{\mathbb{P}(R_{\phi_n}^{\sss(G_n)}>k)}{\mathbb{P}(R_\phi>k)},
\end{equation}
and employ Theorem \ref{thm: upperBoundUndirected} and \eqref{eq: PageRankfraction} to conclude the asymptotic upper bound
\begin{equation}\label{eq: ConfModelUpperBound}
 \mathbb{P}(R_\phi>k) \leq \mathbb{P}(D>k) \qquad\text{ for any } k>0.   
\end{equation}
Now assume that $D$ has a power-law distribution for some $\tau>1$ and some function $\mathcal{L}$ that is slowly varying at infinity. 
Then by \eqref{eq: ConfModelLowBound} and \eqref{eq: ConfModelUpperBound} for every $s \in (0,1)$ there is some $K_s$ such that
\begin{equation}
s \beta^{-(\tau-1)}\mathcal{L}(\beta k)k^{-(\tau-1)} \leq \mathbb{P}(R_\phi>k) \leq \mathcal{L}(k)k^{-(\tau-1)} \quad \text{ for all } k \geq K_s.    
\end{equation}
As $\mathcal{L}$ is slowly varying at infinity, for any $\varepsilon \in (0,1)$ we find a $\hat{K}_\varepsilon \in \N$ such that $\mathcal{L}(\beta k) \geq (1-\varepsilon) \mathcal{L}(k)$ for all $k \geq \hat{K}_\varepsilon$. 
Hence, with 
\begin{equation}
\underline{a}(\varepsilon,s):=s \beta^{-(\tau-1)}\min\{ \ \min\{\mathcal{L}(\beta k) \colon k \in \{0,\ldots, \max\{\hat{K}_\varepsilon,K_s \} \}, 1-\varepsilon \ \} \in (0,1) \end{equation}
we obtain 
\begin{equation}
\underline{a}(s,\varepsilon)\mathcal{L}(k)k^{-(\tau-1)} \leq \mathbb{P}(R_\phi>k) \leq \mathcal{L}(k)k^{-(\tau-1)},    
\end{equation}
and any choice of $s, \varepsilon \in (0,1)$ gives the desired power-law bounds for $R_\phi$.
\end{proof}

\subsection{Application to the P\'olya point tree and the undirected preferential attachment model}
Let us denote by $\mathrm{PA}^{m,\delta}_n$ the version of the preferential attachment model with parameters $m \in \N$ and $\delta>-1$ without self-loops at time $n$, in which we start at time $n=2$ with two vertices with labels $1$ and $2$ and $m$ edges between them and 
at any time step $n \to n+1$ we subsequently add $m$ edges without creating loops. This is done in $m$ sub-steps where at any sub-step the attachment rule is updated according to the new degree distribution. 
The update rule reads (see  \cite[(1.3.65)]{vH24}): for the $(j+1)$th edge of vertex $v_{n+1}$, attached at time-step $n+1$,
\begin{equation*}
\mathbb{P}(v_{j+1,n+1}^{(m)}\conn v_i^{(m)} \mid \mathrm{PA}^{m,\delta}_n)=\frac{D_{i}(n,j)+\delta}{2m(n-1)+j+\delta n}, 
\end{equation*}
where $D_{i}(n,j)$ is the degree of vertex $n$ after the $j$th edge of vertex $v_{n+1}$ has been added.  
Note that in \cite{vH24}, this model is denoted by $\mathrm{PA}^{m,\delta}(d)$ where the $d$ abbreviates \textit{version d} within the class of considered preferential attachment models; also see the erratum to \cite{vH24}.

The preferential attachment model converges locally weakly to the \textit{P\'olya point tree} (\cite[Section 2.3.2]{BeBoCh14}), 
a continuous-time branching process tree to which our machinery applies as well. We stick to the notions employed in \cite[Chapter 5]{vH24}.
The P\'olya point tree is a multi-type branching process, where labels come from the type space $[0,1] \times \{y,o\}$. Here, the first component refers to the \textit{age} of a vertex (or date of its birth), where smaller values mean that the vertex is older. The second component describes whether the vertex is \textit{younger} (y) or \textit{older} (o) with respect to its parent. This distinction is specific to the undirected model and corresponds to the fact that in the preferential attachment model, 
a random vertex has edges to $m$ older neighbors and then starts gaining connections from younger vertices.
As a consequence, the offspring distribution of any vertex has a deterministic component of a fixed number of older vertices and a random component describing the number of younger vertices. In the directed version of the P\'olya point tree in \cite[Definition 4.4]{BaOC22}, the directness of edges makes labelling of vertices superfluous.
Formally, the P\'olya point tree is defined recursively as follows:
\begin{enumerate}
    \item The root $\phi$ has age $U_\phi \sim \mathrm{Unif}([0,1])$ and \textbf{no} label;
    \item For the recursion step consider a vertex $w$ (in the Ulam-Harris representation) with age $A_w$. Then the number $m_-$ of older vertices (i.e., having label $o$) depends on the label of $w$ in that $m_-(w)=
    m-1$ if $w$ has label $y$, while $m_-(w)=
    m$ if $w$ is the root or has label $o$.
    The ages of the $m_-(w)$ older vertices $w1,\ldots,wm_-(w)$ are distributed according to $A_{wj}=U_{wj}^\chi A_w<A_w$-a.s., where $(U_{wj})_{j=1}^{m_-(w)}$ are i.i.d.\ Unif$([0,1])$ random variables independently of everything else and $\chi=\frac{m+\delta}{2m+\delta}<1$.\\
    The younger children of $\omega$ are described by their ages $(A_{w(m_-(w)+j}))_{j \in \mathbb{N}}$, which are the ordered points of a Cox process on $[A_w,1]$ with intensity measure $\rho$ described by the Lebesgue-density
    \begin{equation}\label{eq: IntensityPolya}
    \rho(dx)=\frac{\Gamma_w}{\tau-1}\frac{x^{\frac{1}{\tau-1}-1}}{A_w^{\frac{1}{\tau-1}}}dx.\end{equation}
Here, $\tau=3+\frac{\delta}{m}$ and 
    \[\Gamma_w \sim \begin{cases}
\Gamma(m+\delta+1,1) \quad& \text{if } w \text{ has label }o; \\
\Gamma(m+\delta,1) \quad &\text{else},
\end{cases}\]
where $\Gamma(r,\lambda)$ denotes a Gamma-distributed random variable having density $f(t)=\mathbbm{1}_{[0,\infty)}(t)\lambda^rt^{r-1}{\mathrm e}^{-\lambda t}/\Gamma(r)$.
\end{enumerate}
The P\'olya point tree has degree with a power-law distribution (see \cite[Eq. (8.4.11)]{vdH17}), 
in that there exists a constant $c_{m,\delta}$ such that
\begin{equation*}
\mathbb{P}(d_\phi=k)=c_{m,\delta}k^{-\tau}(1+O(k^{-1})).    
\end{equation*}
In particular, with $\tilde{a}_{m,\delta}:=\frac{c_{m,\delta}}{\tau-1}$, this translates into
\begin{equation*}
\tilde{a}_{m,\delta}k^{-(\tau-1)}(1+O(k^{-1})) \leq \mathbb{P}(d_\phi>k) \leq \tilde{a}_{m,\delta}k^{-(\tau-1)}(1+O(k^{-1})).    
\end{equation*}

\begin{thm}[Power-law hypothesis for undirected preferential attachment models]\label{thm: PrefAttModell}
Consider the undirected preferential attachment model $G_n=PA^{m,\delta}_n$, with $m \geq 1$ and $\delta>-m$.
Let further $c \in (0,1)$. Then
the PageRank vector $\boldsymbol{R}^{\sss(G_n)}$  with damping factor $c$ satisfies
\begin{equation}
\mathbb{P}(R_{\phi_n}^{\sss(G_n)}>k) \leq \mathbb{P}(d_{\phi_n}^{\sss(G_n)}> k) \quad \text{ for all } n \in \N \text{ and all } k \in (0,\infty),
\end{equation}
while further, for any $\beta > 2\lceil 2m+\delta \rceil/[c(1-c)]$,
\begin{equation}
\liminf_{k \rightarrow \infty}\liminf_{n \rightarrow \infty} \frac{\mathbb{P}(R_{\phi_n}^{\sss(G_n)}>k)}{\mathbb{P}(d_{\phi_n}^{\sss(G_n)}>\beta k)} \geq 1.
\end{equation}
In particular, the limiting root-PageRank $R_\phi$ has  power-law tails with the same exponent $\tau=3+\delta/m$ as the degree of the root of the P\'olya point tree. 
\end{thm}
\begin{proof}[Proof of Theorem \ref{thm: PrefAttModell}]
The general structure of the proof is the same as the one of the proof of Theorem \ref{thm: ConfModell}, with the main difference being that verifying the assumptions of Corollary \ref{cor: lowerbound} for the asymptotic lower bound requires a more extensive reasoning. So let us focus on this aspect at first.

The preferential attachment model with parameters $m,\delta$ converges locally weakly to the P\'olya point tree (see \cite[Theorem 2.2]{BeBoCh14} and \cite[Theorem 5.8]{vH24}) with the respective parameters.
We want to verify the Assumption b) of Corollary \ref{cor: lowerbound} with $\boldsymbol{N}_{\phi,k}$ being the set of younger children of the root vertex, ignoring the deterministic number $m$ of older children.
It suffices to stochastically dominate the offspring distribution of the younger children by an i.i.d.\ family of $\N$-valued random variables.
Let us first gather some important properties of the underlying Cox-process:
\begin{lemma}[Degree structure of neighbors of the root in preferential attachment models]
\label{lem: iidBirthtimes}
Let $k>m$ be an integer and $t \in (0,1)$. 
Further let $\pi\colon \{m+1,m+2,\ldots,k\} \rightarrow \{m+1,m+2,\ldots,k\}$ be a permutation chosen uniformly at random and independent of everything else.
Consider the Cox-process $\eta$ describing the birth-times of the younger children of the root with intensity given by \eqref{eq: IntensityPolya}.
\begin{enumerate}[a)]
\item Conditionally on $U_\phi=t$ and $d_\phi=k$, the family $(A_{\pi_j})_{j=m+1}^k$ of birth-times of the $k-m$ younger children of the root is i.i.d.\ $\sim A^{(t)}_{\pi_{m+1}}$ described by the Lebesgue-density 
\begin{equation}\label{eq: fst}
f_{t,\tau}(x)=\mathbbm{1}_{[t,1]}(x)\frac{x^{\frac{1}{\tau-1}-1}}{(\tau-1)(1-t^{\frac{1}{\tau-1}})}.\end{equation}    
\item For every $t \in (0,1)$ the random variable $A^{(t)}_{\pi_{m+1}}$ is stochastically bounded from below by the random variable $A^{(0)}_{\pi_{m+1}}$ described by the Lebesgue-density 
\eqn{f_{0,\tau}(x)=\mathbbm{1}_{[0,1]}(x)\frac{x^{\frac{1}{\tau-1}-1}}{\tau-1}.}
\end{enumerate}
\end{lemma}
\begin{proof}[Proof of Lemma \ref{lem: iidBirthtimes}]
Statement a) is a folklore property of Poisson- and Cox-processes, noting that the random scalar factor $\Gamma_\phi$ in the intensity measure $\rho$ in \eqref{eq: IntensityPolya} does not exhibit any temporal dependence.
In more detail, let $\rho_{y,t}$ denote the intensity measure $\rho$ of the process $\eta$ in \eqref{eq: IntensityPolya} conditioned on $\Gamma_\phi=y$ and $U_\phi=t$, which is an inhomogeneous Poisson process. By standard arguments (e.g.\ \cite[Proposition 3.5]{LaPe18}), this Poisson process is distributed as a mixed binomial process on $[t,1]$ with sampling distribution described by the density $f_{t,\tau}(x)=\frac{\rho_{y,t}(dx)}{\rho_{y,t}([t,1])}$ and mixing distribution $\mathrm{Poi}(\rho_{y,t}([t,1]))$.
Hence, we have that conditional on $U_\phi=t$, $d_\phi=k$ and $\Gamma_\phi=y$ the family $(A_{\pi_j})_{j=m+1}^k$ is i.i.d. \  $\sim A^{(t)}_{\pi_{m+1}}$ with density $f_{t,\tau}$, which does not depend on $y$. In particular, for any measurable $g\colon \R^{k-m} \rightarrow \R$,
\[\mathbb{E}\Big[g\left((A_{\pi_j})_{j=m+1}^k \right) \mid d_\phi=k, U_\phi, \Gamma_\phi\Big]= \mathbb{E}\Big[g\left((A_{\pi_j})_{j=m+1}^k \right) \mid d_\phi=k, U_\phi\Big] \quad \text{a.s.},\]
because by the factorization Lemma (e.g.\ \cite[Corollary 1.97]{Kl20}) and the fact that $f_{t,\tau}$ does not depend on $y$ the l.h.s.\ is already measurable w.r.t.\ $\sigma(\{d_\phi=k\},U_\phi) \subseteq \sigma(\{d_\phi=k\},U_\phi,\Gamma_\phi)$.
This implies that, for all Borel-measurable $B_{m+1}, \ldots, B_k$,
\begin{equation}
\begin{split}
\mathbb{P}\left(\bigcap_{j=m+1}^k \{A_{\pi_j} \in B_j \} \mid d_\phi=k, \, U_\phi \right)&=\mathbb{P}\left(\bigcap_{j=m+1}^k \{A_{\pi_j} \in B_j \} \mid d_\phi=k, \, U_\phi, \, \Gamma_\phi \right) \cr 
=\prod_{j=m+1}^k\mathbb{P} \left(A_{\pi_j} \in B_j\mid d_\phi=k, \, U_\phi, \, \Gamma_\phi \right)
&=\prod_{j=m+1}^k\mathbb{P}\left(A_{\pi_j} \in B_j\mid d_\phi=k, \, U_\phi \right) \quad \text{a.s.} 
\end{split}
\end{equation}
In particular, the family $(A_{\pi_j})_{j \in \{m+1,m+2,\ldots,k\}}$ is independent conditional on $d_\phi=k$ and $U_\phi$.

To prove Statement b), we note that, for every $t \in (0,1)$ and all $x \in \R$,
\begin{equation}
\begin{split}
\mathbb{P}(A^{(t)}_{\pi_{m+1}} \geq x)&=\mathbbm{1}_{(-\infty, t)}(x)+\mathbbm{1}_{[t,1]}(x)\frac{1-x^\frac{1}{\tau-1}}{1-t^\frac{1}{\tau-1}} \cr 
&\geq \mathbbm{1}_{(-\infty, 0)}(x)+\mathbbm{1}_{[0,1]}(x)(1-x^\frac{1}{\tau-1})=\mathbb{P}(A^{(0)}_{\pi_{m+1}} \geq x).
\end{split}
\end{equation}
Hence, $A^{(0)}_{\pi_{m+1}} \preceq A^{(t)}_{ \pi_{m+1}}$.
\end{proof}
\begin{lemma}[I.i.d.\ structure of younger neighbors in preferential attachment models]\label{lem: iidDegreesPolya}
Conditional on $d_\phi=k$ and $U_\phi=t$, the degrees  $(d_{\pi_j})_{j \in \{m+1,m+2,\ldots,k\}}$ of the younger children of the root are i.i.d., where we have for every $n \geq m$
\[\mathbb{P}(d_{\pi_{m+1}}=n \mid d_\phi=k, U_\phi=t)=
\int_t^1\mathbb{P}(d_{\pi_{m+1}}=n \mid A_{\pi_{m+1}}=x)\mathbb{P}^{A_{\pi_{m+1}}^{(t)}}(dx),
\]
independently of $k$. Here, $\mathbb{P}^{A_{\pi_{m+1}}^{(t)}}$ denotes the distribution of $A_{\pi_{m+1}}$ conditional on $U_\phi=t$ and $d_\phi=k$.
Moreover, conditional on $d_\phi=k$ and $U_\phi=t,$ 
the degrees  $(d_{\pi_j})_{j \in \{m+1,m+2,\ldots,k\}}$ are stochastically bounded from above by the i.i.d.-family 
$(\tilde{d}_{j})_{j \in \{m+1,m+2,\ldots,k\}}$ with, for every $n \geq m$,
\begin{equation}\label{eq: degIntegral}
\begin{split}
\mathbb{P}(\tilde{d}_{m+1}=n)&= 
\int_0^1\mathbb{P}(d_{\pi_{m+1}}=n \mid A_{\pi_{m+1}}=x)\mathbb{P}^{A_{\pi_{m+1}}^{(0)}}(dx)
\end{split}
\end{equation}
independently of $k$ and $t$.
In particular, for every $n \geq m$,
\begin{equation}
\mathbb{P}(\tilde{d}_{m+1}=n)=\frac{m+\delta}{(n+\delta)(n+\delta+1)}.    
\end{equation}
\end{lemma}

\begin{proof}[Proof of Lemma \ref{lem: iidDegreesPolya}]
We start by recalling two properties that directly follow from the definition of the P\'olya point tree:
\begin{itemize}
    \item[$\rhd$] Conditionally on $d_\phi=k$ and $(A_{\pi_j})_{j=m+1}^{k}$, the random variables $d_{\pi_{m+1}},d_{\pi_{m+2}}, \ldots, d_{\pi_k}$ describe the total number of points of independent Cox-processes and are thus independent random variables.
    \item[$\rhd$] Further conditioning on $U_\phi$ does not add any information. More precisely, for any measurable $g\colon \R^{k-m} \rightarrow \R$ the conditional expectation of $g((d_{\pi_j})_{j=m+1}^{k})$ given $d_\phi=k$, $(A_{\pi_j})_{j=m+1}^{k}$ \textbf{and additionally} $U_\phi$ is already measurable w.r.t. $d_\phi=k$ and $(A_{\pi_j})_{j=m+1}^{k}$.
    \end{itemize}
By these two properties, we conclude that, for any integers $n_{m+1}, \ldots,n_{k}$,
\begin{equation}\label{eq: IndOffspring1}
\begin{split}
&\mathbb{P}(d_{\pi_{m+1}}=n_{m+1}, \ldots, d_{\pi_{k}}=n_{k} \mid d_\phi=k, U_\phi=t)\cr
&=\int_{(0,1)^{k-m}} \mathbb{P}(d_{\pi_{m+1}}=n_{m+1}, \ldots, d_{\pi_{k}}=n_{k} \mid (A_{\pi_j})_{j=m+1}^{k}=(x_{j})_{j=m+1}^{k},d_\phi=k, U_\phi=t) \cr 
&\qquad\qquad \times\mathbb{P}(A_{\pi_{m+1}} \in dx_{m+1}, \ldots, A_{\pi_{k}} \in dx_{k} \mid d_\phi=k, U_\phi=t)\cr 
&=\int_{(0,1)^{k-m}} \mathbb{P}(d_{\pi_{m+1}}=n_{m+1}, \ldots, d_{\pi_{k}}=n_{k} \mid (A_{\pi_j})_{j=m+1}^{k}=(x_{j})_{j=m+1}^{k},d_\phi=k) \cr 
&\qquad\qquad \times\mathbb{P}(A_{\pi_{m+1}} \in dx_{m+1}, \ldots, A_{\pi_{k}} \in dx_{k} \mid d_\phi=k, U_\phi=t)\cr 
&=\int_{(0,1)^{k-m}} \prod_{j=m+1}^k\mathbb{P}(d_{\pi_{j}}=n_{j} \mid (A_{\pi_j})_{j=m+1}^{k}=(x_{j})_{j=m+1}^{k},d_\phi=k) \cr 
&\qquad \qquad \times
\mathbb{P}(A_{\pi_{m+1}} \in dx_{m+1}, \ldots, A_{\pi_{k}} \in dx_{k} \mid d_\phi=k, U_\phi=t).\cr 
\end{split}
\end{equation}
By construction of the P\'olya point tree, the conditional distribution of each random variable $d_{\pi_j}$ given $(A_{\pi_l})_{l \in \{m+1,\ldots,k\}}$ and $d_\phi=k$ is already measurable w.r.t.\ $A_{\pi_j}$. Hence, the last expression of \eqref{eq: IndOffspring1} equals
\begin{equation}\label{eq: IndOffspring2}
\begin{split}
&\int_{(0,1)^{k-m}} \prod_{j=m+1}^k\mathbb{P}(d_{\pi_{j}}=n_{j} \mid A_{\pi_j}=x_j)\mathbb{P}(A_{\pi_{m+1}} \in dx_{m+1}, \ldots, A_{\pi_{k}} \in dx_{k} \mid d_\phi=k, U_\phi=t).\cr 
\end{split}
\end{equation}
Now employ Lemma \ref{lem: iidBirthtimes} to conclude that $A_{\pi_{m+1}}, \ldots, A_{\pi_{k}}$ are i.i.d.\ given $ d_\phi=k, U_\phi=t$. By Fubini's Theorem, \eqref{eq: IndOffspring2} thus equals
\begin{equation}
\begin{split}
&\int_{(0,1)}\mathbb{P}(A_{\pi_{m+1}} \in dx_{m+1} \mid d_\phi=k, U_\phi=t) \ldots \int_{(0,1)}\mathbb{P}(A_{\pi_{k}} \in dx_{k} \mid d_\phi=k, U_\phi=t) \cr 
&\qquad\qquad \times \prod_{j=m+1}^k\mathbb{P}(d_{\pi_{j}}=n_{j} \mid A_{\pi_j}=x_j)\cr
&= \prod_{j=m+1}^k \int_{(0,1)}\mathbb{P}(d_{\pi_{j}}=n_{j} \mid A_{\pi_j}=x_j)\mathbb{P}(A_{\pi_{j}} \in dx_{j} \mid d_\phi=k, U_\phi=t). 
\end{split}
\end{equation}
This proves the conditional independence with, for every $n \geq m$,
\begin{equation}\label{eq: degDistrCond}
\begin{split}
 &\mathbb{P}(d_{\pi_{m+1}}=n \mid d_\phi=k, U_\phi=t)=\int_{t}^1 \mathbb{P}(d_{\pi_{m+1}}=n \mid A_{\pi_{m+1}}=x)\mathbb{P}^{A_{\pi_{m+1}}^{(t)}}(dx). 
 \end{split}
\end{equation}
Next, for the stochastic bound \eqref{eq: degIntegral}, we first note that by \eqref{eq: IntensityPolya}, for any $0<\underline{x} \leq \bar{x}<1$, the intensity $\rho([A_{\pi_{m+1}},1])$ conditionally on $A_{\pi_{m+1}}=\bar{x}$ is stochastically bounded from above by the intensity $\rho([A_{\pi_{m+1}},1])$ conditionally on $A_{\pi_{m+1}}=\underline{x}$. Hence, also $d_{\pi_{m+1}}$ conditionally on $A_{\pi_{m+1}}=\bar{x}$ is stochastically bounded from above by 
 $d_{\pi_{m+1}}$ conditionally on $A_{\pi_{m+1}}=\underline{x}$.
This means that, for every fixed $n \geq m$, the function $g_n\colon (0,1) \rightarrow [0,1]; \  x \mapsto \mathbb{P}(d_{\pi_{m+1}} < n \mid A_{\pi_{m+1}}=x)$ is non-decreasing.
 Hence, from $A_{\pi_{m+1}}^{(0)} \preceq A_{\pi_{m+1}}^{(t)}$ it follows $\mathbb{E}[g_n(A_{\pi_{m+1}}^{(0)})] \leq \mathbb{E}[g_n(A_{\pi_{m+1}}^{(t)})]$ for every $t \in (0,1)$ (e.g., \cite[Theorem 1.2.2]{St83} or \cite[Equation 1.A.7]{ShSh07}). This
implies that $\mathbb{P}(\tilde{d}_{m+1}<n) \leq \mathbb{P}(d_{\pi_{m+1}} < n \mid d_\phi=k, U_\phi=t)$ for all possible choices of $k,n$ and $t$ and thus proves the proposed stochastic bound. 

To calculate the probability mass function of $\tilde{d}_{m+1}$, we insert the expression
\[
\mathbb{P}(d_{\pi_{m+1}}=n \mid A_{\pi_{m+1}}=x)= \frac{\Gamma(n+\delta)}{(n-m)!\Gamma(m+\delta)} \left(1-x^\frac{1}{\tau-1} \right)^{n-m} x^{\frac{m+\delta}{\tau-1}}
\] as given in \cite[Equation 5.3.10]{vH24} and the density $f_{0,\tau}(x)=\frac{x^{2-\tau}}{\tau-1}$ into \eqref{eq: degIntegral}, which gives
\begin{equation}\label{eq: degIntegral2}
\mathbb{P}(\tilde{d}_{m+1}=n)=\frac{1}{\tau-1}\frac{\Gamma(n+\delta)}{(n-m)!\Gamma(m+\delta)} \int_0^1 \left(1-x^\frac{1}{\tau-1} \right)^{n-m} x^{\frac{m+\delta+2-\tau}{\tau-1}}dx.
\end{equation}
Substituting $u=x^\frac{1}{\tau-1}$, i.e., $dx=(\tau-1)u^{\tau-2}$ in the integral in \eqref{eq: degIntegral2} we arrive at
\begin{equation}
\begin{split}
\int_0^1 \left(1-x^\frac{1}{\tau-1} \right)^{n-m} x^{(m+\delta+2-\tau)/(\tau-1)}dx&=(\tau-1)\int_0^1u^{m+\delta}(1-u)^{n-m}du \cr&=(\tau-1)\frac{\Gamma(m+\delta+1)\Gamma(n-m+1)}{\Gamma(n+\delta+2)},
\end{split}
\end{equation}
which leads to
\begin{equation}\label{eq: TermsDegreesPolya}
\begin{split}
\mathbb{P}(\tilde{d}_{m+1}=n)&=\frac{\Gamma(n+\delta)}{(n-m)!\Gamma(m+\delta)}\frac{\Gamma(m+\delta+1)\Gamma(n-m+1)}{\Gamma(n+\delta+2)} \cr
&=\frac{\Gamma(n+\delta)\Gamma(m+\delta+1)}{\Gamma(m+\delta)\Gamma(n+\delta+2)}=\frac{m+\delta}{(n+\delta)(n+\delta+1)}.  
\end{split}
\end{equation}
\end{proof}

\noindent
\textbf{Continuation of the proof of Theorem \ref{thm: PrefAttModell}.}
As its distribution does not explicitly depend on the ``internal" random variable $U_\phi$, the family $(\tilde{d}_j)_{j \in \{m+1,m+2, \ldots,k\}}$ is a suitable stochastic upper bound in the sense of Corollary $\ref{cor: lowerbound}$ b). 
In particular, we can condition on $U_\phi$ in the second inequality of \eqref{eq: ConsequenceOfMarokov} and similarly to \eqref{eq: TotalExpUB} employ total expectation. 
For the scaling factor $\beta$ we need to calculate $\min\{t \in \N \mid \mathbb{P}(\tilde{d}_1 \geq t) \leq 1/2\}$. 
By using a telescoping sum identity, we obtain
\begin{equation}
\begin{split}
\mathbb{P}(\tilde{d}_{m+1} \geq t)&=(m+\delta)\sum_{n=t}^\infty \frac{1}{(n+\delta)(n+\delta+1)}=(m+\delta)\sum_{n=t}^\infty \left(\frac{1}{n+\delta}-\frac{1}{(n+\delta)+1}\right) \cr
&=\frac{m+\delta}{t+\delta},    
\end{split}    
\end{equation}
hence $\mathbb{P}(\tilde{d}_{m+1} \geq t) \leq 1/2$ is equivalent to $t \geq \lceil 2m+\delta \rceil$.
From Corollary \ref{cor: lowerbound} b) we thus obtain that for every $\beta>2\lceil 2m+\delta \rceil$ 
\begin{equation*} 
\mathbb{P}(R_\phi>k) \geq (1+o(1)) \mathbb{P}(d_\phi>\beta k) \quad \text{ as }k \rightarrow \infty.  
\end{equation*}
The remainder of the proof of Theorem \ref{thm: PrefAttModell} is a minor modification of that of Theorem \ref{thm: ConfModell}. 
\end{proof}
\section{A counterexample on the lower bound}\label{sec: counterexample}
Condition \eqref{eq: LowerBoundMainAssumtion} in Proposition \ref{pro: PageRankLowerBound} can be violated once the dependencies between the degrees of the neighbors of a vertex become too large, as we show in this section. 
We show that, for any prescribed power-law distribution $p=(p_k)_{k \in \N}$, there exists a sequence $(\tilde{G}_n^{(p)})_{n \geq n_p}$, $n_p \in \N$, of connected graphs whose degree distribution converges to $p$, whereas the corresponding PageRank at a uniformly chosen root converges to one in probability. This implies that there is no general non-trivial asymptotic lower bound on the PageRank vector.

Our construction follows a two-step procedure: First we construct \textbf{disconnected} graphs $G_n$ described as the disjoint union of $o(n)$ many degree-regular graphs of appropriate size depending on $(p_k)_{k \in \N}$. The corresponding PageRank-matrix $\boldsymbol{P}^{(\sss G_n)}$ is doubly-stochastic, and therefore the PageRank vector of such graph has all elements equal to 1. This step is explained in Section \ref{Subsec: Disjoint union}. Afterwards we connect the formerly disjoint components by adding $o(n)$ edges. After showing that our specific choice of $G_n$ converges locally weakly, we then employ Theorem \ref{thm: localConvPageRank} and uniqueness of the local weak limit to conclude that the PageRank associated with the so-obtained \textbf{connected} graphs converges in probability to $1$. This step is explained in Section \ref{subsec: Connecting components}. 
\subsection{Disjoint unions of regular graphs}\label{Subsec: Disjoint union}
Let us start by formalizing the above ideas:
\begin{defn}[Degree-regular circulant graph]\label{def: regular circulant}
Let $2 \leq k < n$ be integers such that $k$ is even. Then let $G_{k,n}$ denote the $k$-regular graph with $n$ vertices which is constructed as follows (see the graphs $G_{2,10}$, $G_{4,8}$ and $G_{6,7}$ in Figure \ref{fig: regular circulant} for an illustration).
\begin{itemize}
\item[$\rhd$] The set of vertices is identified with the one-dimensional torus $\Z_n =\Z/ n\Z$. 
\item[$\rhd$]   Every vertex $i \in \Z_n$ is connected by an edge to the vertices $i \pm 1, i \pm2, \ldots, i \pm k/2$.
\end{itemize}
The graph is {\em circulant} in that cyclic permutations of its vertices are graph-automorphisms.
\end{defn}

\begin{defn}[Disjoint union of degree-regular graphs]\label{def: disjoint circulant graphs}
Let $n \geq 2$. 
Let $p=(p_k)_{k \in \N}$ be any probability mass function such that $p_k=0$ for all odd $k$, and set $N_{k,n}:=N^{(p)}_{k,n}:=\lfloor np_{k}\rfloor.$
Let $(M_n)_{n \in \N}$ be sequence in $\N$ that diverges to infinity slowly enough, so that it in particular satisfies that
\begin{equation}
N_{k,n} \geq k+1 \quad \text{for all } k \in [M_n].
\end{equation}
Let $n_{p} :=\min \{ n \in \N \colon M_n \geq 2\}$.
For every $n \geq n_p$ let $G_n^{(p)}:=\coprod_{k \in [M_n]} G_{k,N_{k,n}}$ denote the graph constructed as the non-empty disjoint union of the circulant graphs \\$G_{2,N_{2,n}}, \ldots, G_{M_n,N_{M_n,n}}$, where we note that $\# G_n^{(p)}=n(1+o(1))$ as $n \rightarrow \infty$.\\
Further, we obtain $\tilde{G}^{(p)}_n$ by connecting two formerly disjoint $G_{k,N_{k,n}}$ whose degrees differ minimally, to turn $G_n^{(p)}$ into a connected graph $\tilde{G}^{(p)}_n$. As each component of $G_n^{(p)}$ is invariant under all cyclic permutations, the concrete choice of these $M_n-1$ edges is irrelevant up to isomorphisms.
\end{defn}
\begin{figure}
\centering
\begin{tikzpicture}[x=1cm,y=1cm]
\node (b1) at (canvas polar cs:angle=90,radius=2cm) [circle, draw=black,thick,fill=yellow,minimum size = 4mm, inner sep = 0]{};
\node (b2) at (canvas polar cs:angle=126,radius=2cm) [circle, draw=black,thick,fill=yellow,minimum size = 4mm, inner sep = 0]{};
\node (b3) at (canvas polar cs:angle=162,radius=2cm) [circle, draw=black,thick,fill=yellow,minimum size = 4mm, inner sep = 0]{};
\node (b4) at (canvas polar cs:angle=198,radius=2cm) [circle, draw=black,thick,fill=yellow,minimum size = 4mm, inner sep = 0]{};
\node (b5) at (canvas polar cs:angle=234,radius=2cm) [circle, draw=black,thick,fill=yellow,minimum size = 4mm, inner sep = 0]{};
\node (b6) at (canvas polar cs:angle=270,radius=2cm) [circle, draw=black,thick,fill=yellow,minimum size = 4mm, inner sep = 0]{};
\node (b7) at (canvas polar cs:angle=306,radius=2cm) [circle, draw=black,thick,fill=yellow,minimum size = 4mm, inner sep = 0]{};
\node (b8) at (canvas polar cs:angle=342,radius=2cm) [circle, draw=black,thick,fill=yellow,minimum size = 4mm, inner sep = 0]{};
\node (b9) at (canvas polar cs:angle=18,radius=2cm) [circle, draw=black,thick,fill=yellow,minimum size = 4mm, inner sep = 0]{};
\node (b10) at (canvas polar cs:angle=54,radius=2cm) [circle, draw=black,thick,fill=yellow,minimum size = 4mm, inner sep = 0]{};
\foreach \from/\to in 
{b1/b2,b2/b3,b3/b4,b4/b5,b5/b6,b6/b7,b7/b8,b8/b9,b9/b10,b10/b1}
\draw[] (\from) to (\to);
\coordinate (c1a) at (canvas polar cs:angle=90,radius=2cm); 
\coordinate (c2a) at (canvas polar cs:angle=135,radius=2cm);
\coordinate (c3a) at (canvas polar cs:angle=180,radius=2cm);
\coordinate (c4a) at (canvas polar cs:angle=225,radius=2cm);
\coordinate (c5a) at (canvas polar cs:angle=270,radius=2cm);
\coordinate (c6a) at (canvas polar cs:angle=315,radius=2cm);
\coordinate (c7a) at (canvas polar cs:angle=0,radius=2cm);
\coordinate (c8a) at (canvas polar cs:angle=45,radius=2cm);
\node (c1) at ($ (c1a) +(5,0)$) [circle, draw=black,thick,fill=yellow,minimum size = 4mm, inner sep = 0]{};
\node (c2) at ($ (c2a) +(5,0)$) [circle, draw=black,thick,fill=yellow,minimum size = 4mm, inner sep = 0]{};
\node (c3) at ($ (c3a) +(5,0)$) [circle, draw=black,thick,fill=yellow,minimum size = 4mm, inner sep = 0]{};
\node (c4) at ($ (c4a) +(5,0)$) [circle, draw=black,thick,fill=yellow,minimum size = 4mm, inner sep = 0]{};
\node (c5) at ($ (c5a) +(5,0)$) [circle, draw=black,thick,fill=yellow,minimum size = 4mm, inner sep = 0]{};
\node (c6) at ($ (c6a) +(5,0)$) [circle, draw=black,thick,fill=yellow,minimum size = 4mm, inner sep = 0]{};
\node (c7) at ($ (c7a) +(5,0)$) [circle, draw=black,thick,fill=yellow,minimum size = 4mm, inner sep = 0]{};
\node (c8) at ($ (c8a) +(5,0)$) [circle, draw=black,thick,fill=yellow,minimum size = 4mm, inner sep = 0]{};
\coordinate (d1a) at (canvas polar cs:angle=90,radius=2cm); 
\coordinate (d2a) at (canvas polar cs:angle=141.43,radius=2cm);
\coordinate (d3a) at (canvas polar cs:angle=192.86,radius=2cm);
\coordinate (d4a) at (canvas polar cs:angle=244.29,radius=2cm);
\coordinate (d5a) at (canvas polar cs:angle=295.71,radius=2cm);
\coordinate (d6a) at (canvas polar cs:angle=347.14,radius=2cm);
\coordinate (d7a) at (canvas polar cs:angle=38.57,radius=2cm);
\node (d1) at ($ (d1a) +(10,0)$) [circle, draw=black,thick,fill=yellow,minimum size = 4mm, inner sep = 0]{};
\node (d2) at ($ (d2a) +(10,0)$) [circle, draw=black,thick,fill=yellow,minimum size = 4mm, inner sep = 0]{};
\node (d3) at ($ (d3a) +(10,0)$) [circle, draw=black,thick,fill=yellow,minimum size = 4mm, inner sep = 0]{};
\node (d4) at ($ (d4a) +(10,0)$) [circle, draw=black,thick,fill=yellow,minimum size = 4mm, inner sep = 0]{};
\node (d5) at ($ (d5a) +(10,0)$) [circle, draw=black,thick,fill=yellow,minimum size = 4mm, inner sep = 0]{};
\node (d6) at ($ (d6a) +(10,0)$) [circle, draw=black,thick,fill=yellow,minimum size = 4mm, inner sep = 0]{};
\node (d7) at ($ (d7a) +(10,0)$) [circle, draw=black,thick,fill=yellow,minimum size = 4mm, inner sep = 0]{};
\foreach \from/\to in 
{b1/b2,b2/b3,b3/b4,b4/b5,b5/b6,b6/b7,b7/b8,b8/b9,b9/b10,b10/b1}
\draw[] (\from) to (\to);
\foreach \from/\to in 
{c1/c2,c1/c3,c2/c3,c2/c4,c3/c4,c3/c5,c4/c5,c4/c6,c5/c6,c5/c7,c6/c7,c6/c8,c7/c8,c7/c1,c8/c1,c8/c2}
\draw[] (\from) to (\to);
\foreach \from/\to in 
{d1/d2,d1/d3,d1/d4,d2/d3,d2/d4,d2/d5,d3/d4,d3/d5,d3/d6,d4/d5,d4/d6,d4/d7,d5/d6,d5/d7,d5/d1,d6/d7,d6/d1,d6/d2,d7/d1,d7/d2,d7/d3}
\draw[] (\from) to (\to);
\draw[dotted, thick, bend left] (b10) to (c2);
\draw[dotted, thick, bend left] (c8) to (d2);
\node (an1) at (0,-3) {$G_{2,10}$};
\node (an2) at (5,-3) {$G_{4,8}$};
\node (an3) at (10,-3) {$G_{6,7}$};
\end{tikzpicture}
    \caption{An illustration for Definitions \ref{def: regular circulant} and \ref{def: disjoint circulant graphs}: The graph $G_{n}$ with $n=25$ vertices corresponding to the vector $(N_{2,25},N_{4,25},N_{6,25})=(10,8,7)$ consists of three disjoint regular circulant graphs, which are constructed according to Definition \ref{def: regular circulant}, of sizes $10,8$ and $7$ and respective degrees $2,4$ and $6$.
To make the graph connected, in the second step we add the two dashed edges to obtain the connected graph $\tilde{G}_{25}$.}
\label{fig: regular circulant}
\end{figure}
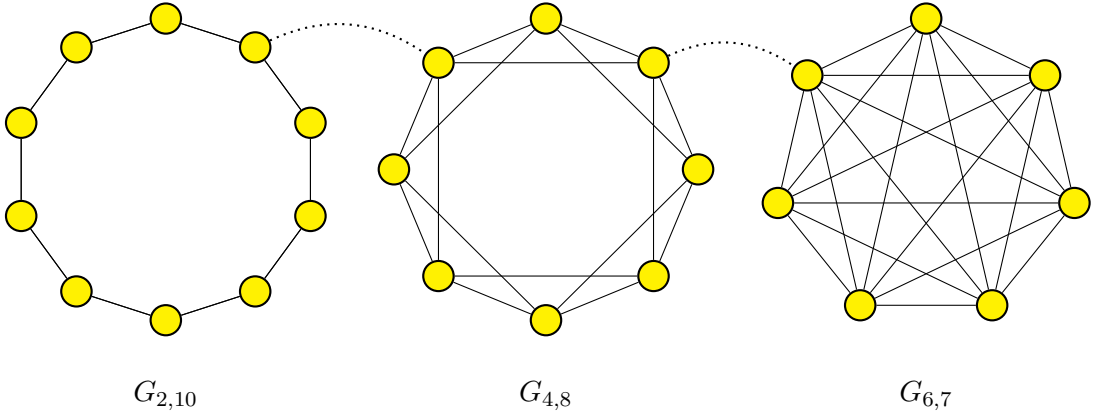
The sequence of disconnected graphs constructed in Definition \ref{def: disjoint circulant graphs} above has the desired properties:
\begin{pro}[Degree and PageRank structure union of disconnected graphs]
\label{pro: disjoint regular}
Let $p=(p_k)_{k \in \N}$ be any probability mass function such that $p_k=0$ for all odd $k$ and $X^{(p)}$ be distributed according to $p$. Then for the graphs $(G_n^{(p)})_{n \geq n_p}$ described in Definition \ref{def: disjoint circulant graphs},
\begin{enumerate}[(a)]
    \item $\lim_{n \rightarrow \infty}\mathbb{P}(d_{\phi_n}^{\sss (G_n^{(p)})}=k)=p_k$.
    \item $R_v^{\sss(G_n^{(p)})}=1$ for every $n \geq n_p$ and $v \in V(G_n^{(p)})$.
\end{enumerate}
\end{pro}
\begin{proof}[Proof of Proposition \ref{pro: disjoint regular}]
Statement (a) follows directly from the construction of the graphs $G^{(p)}_n$.
Statement (b) is a consequence of the fact that the corresponding PageRank-matrix $\boldsymbol{P}^{(\sss G_n^{(p)})}$ is doubly-stochastic.
\end{proof}
\begin{rk}\label{rk: regular circulant}
{\normalfont \rmfamily We may replace the circulant regular graphs $G_{k,n}$ in Definition \ref{def: disjoint circulant graphs} by any degree-regular graphs of the prescribed degree and size provided that the sequence $(G_n^{(p)})_{n \in \N}$ converges in the local weak sense. The concrete choice from Definition \ref{def: regular circulant} above provides a specific illustrative example for which the local weak convergence is easy to prove (see Proposition \ref{pro: LWconv} below), at the price that only regular graphs of even degree occur. }\hfill\ensymboldefinition
\end{rk}
\subsection{Connecting the disjoint components}\label{subsec: Connecting components}
We next study the local weak limits of $G_n^{(p)}$ and $\tilde{G}^{(p)}_n$:

\begin{pro}[Local limit of disjoint unions of circulant graphs, and their connected versions]
\label{pro: LWconv}
Let $p=(p_k)_{k \in \N}$ be any degree distribution such that $p_k=0$ for all odd $k$. Then, $(G_n^{(p)})_{n \in \N}$ described in Definition \ref{def: disjoint circulant graphs} converges locally weakly to a rooted graph $(G_\infty^{(p)},\phi)$. In particular, also the sequence $(\tilde{G}^{(p)}_n)_{n \in \N}$ of connected graphs converges locally weakly to the same limit $(G_\infty^{(p)},\infty)$. 
\end{pro}

\begin{proof}[Proof of Proposition \ref{pro: LWconv}]
By Remark \ref{rk: LWCriterion} above, 
for fixed $k \in 2\N$, the sequence \\$(G_{k,n})_{n \geq n_p, M_n \geq k}$ converges locally weakly to the graph $(G_{k,\infty},\phi)$ on $\Z$ with $\phi \equiv 0$, where two vertices in $G_{k,\infty}$ are connected by an edge iff their distance is at most $k/2$ (see Figure \ref{fig: isomorphic} for the specific graph $G_{4,\infty}$).
Hence, the sequence $G_n^{(p)}$ converges locally weakly to the random rooted graph $(G_\infty^{(p)},0)$ which equals $(G_{k,\infty},0)$ with probability $\lim_{n \rightarrow \infty}N_{k,n}/n=p_k$.

As each component of $G_n^{(p)}$ is invariant under all cyclic permutations, the concrete choice of these $M_n-1$ edges is irrelevant up to isomorphisms. The  operation of adding edges affects the degree of $o(n)$ vertices and thus, by Remark \ref{rk: LWCriterion}, does not change the local weak limit.
\begin{figure}
\centering
\begin{tikzpicture}[x=1cm,y=1cm]
\node (b1) at (canvas polar cs:angle=90,radius=2cm) [circle, draw=black,thick,fill=yellow,minimum size = 4mm, inner sep = 0]{};
\node (b2) at (canvas polar cs:angle=126,radius=2cm) [circle, draw=black,thick,fill=yellow,minimum size = 4mm, inner sep = 0]{};
\node (b3) at (canvas polar cs:angle=162,radius=2cm) [circle, draw=black,thick,fill=yellow,minimum size = 4mm, inner sep = 0]{};
\node (b4) at (canvas polar cs:angle=198,radius=2cm) [circle, draw=black,thick,fill=yellow,minimum size = 4mm, inner sep = 0]{};
\node (b5) at (canvas polar cs:angle=234,radius=2cm) [circle, draw=black,thick,fill=yellow,minimum size = 4mm, inner sep = 0]{};
\node (b6) at (canvas polar cs:angle=270,radius=2cm) [circle, draw=black,thick,fill=yellow,minimum size = 4mm, inner sep = 0]{};
\node (b7) at (canvas polar cs:angle=306,radius=2cm) [circle, draw=black,thick,fill=yellow,minimum size = 4mm, inner sep = 0]{};
\node (b8) at (canvas polar cs:angle=342,radius=2cm) [circle, draw=black,thick,fill=yellow,minimum size = 4mm, inner sep = 0]{};
\node (b9) at (canvas polar cs:angle=18,radius=2cm) [circle, draw=black,thick,fill=yellow,minimum size = 4mm, inner sep = 0]{};
\node (b10) at (canvas polar cs:angle=54,radius=2cm) [circle, draw=black,thick,fill=yellow,minimum size = 4mm, inner sep = 0]{};
\node (ya1) at (2,0) {};
\node (ya2) at (2.5,0) {};
\node (ya3) at (3.0,0) {};
\node (ya4) at (11.5,0) {};
\node (ya5) at (12,0) {};
\node (ya6) at (12.5,0) {};
\node (c1) at (3.5,0) [circle, draw=black,thick,fill=red,minimum size = 3mm, inner sep = 0]{};
\node (c2) at (4,0) [circle, draw=black,thick,fill=red,minimum size = 3mm, inner sep = 0]{};
\node (c3) at (4.5,0) [circle, draw=black,thick,fill=red,minimum size = 3mm, inner sep = 0]{};
\node (c4) at (5,0) [circle, draw=black,thick,fill=red,minimum size = 3mm, inner sep = 0]{};
\node (c5) at (5.5,0) [circle, draw=black,thick,fill=red,minimum size = 3mm, inner sep = 0]{};
\node (c6) at (6,0) [circle, draw=black,thick,fill=red,minimum size = 3mm, inner sep = 0]{};
\node (c7) at (6.5,0) [circle, draw=black,thick,fill=red,minimum size = 3mm, inner sep = 0]{};
\node (c8) at (7,0) [circle, draw=black,thick,fill=red,minimum size = 3mm, inner sep = 0]{0};
\node (c9) at (7.5,0) [circle, draw=black,thick,fill=red,minimum size = 3mm, inner sep = 0]{};
\node (c10) at (8,0) [circle, draw=black,thick,fill=red,minimum size = 3mm, inner sep = 0]{};
\node (c11) at (8.5,0) [circle, draw=black,thick,fill=red,minimum size = 3mm, inner sep = 0]{};
\node (c12) at (9,0) [circle, draw=black,thick,fill=red,minimum size = 3mm, inner sep = 0]{};
\node (c13) at (9.5,0) [circle, draw=black,thick,fill=red,minimum size = 3mm, inner sep = 0]{};
\node (c14) at (10,0) [circle, draw=black,thick,fill=red,minimum size = 3mm, inner sep = 0]{};
\node (c15) at (10.5,0) [circle, draw=black,thick,fill=red,minimum size = 3mm, inner sep = 0]{};
\node (c16) at (11,0) [circle, draw=black,thick,fill=red,minimum size = 3mm, inner sep = 0]{};
\node (an1) at (11.2,-0.5) {$\Z$};
\foreach \from/\to in 
{b1/b2,b1/b3,b2/b3,b2/b4,b3/b4,b3/b5,b4/b5,b4/b6,b5/b6,b5/b7,b6/b7,b6/b8,b7/b8,b7/b9,b8/b9,b8/b10,b9/b10,b9/b1,b10/b1,b10/b2}
\draw[] (\from) to (\to);
\draw[rounded corners,dotted,rotate around={33:(b2)}] (-3,0.7) rectangle (0.5,2.2);
\draw[rounded corners,dotted] (6.25,-0.5) rectangle (7.75,0.5);
\foreach \from/\to in 
{c1/c2,c2/c3,c3/c4,c4/c5,c5/c6,c6/c7,c7/c8,c8/c9,c9/c10,c10/c11,c11/c12,c12/c13,c13/c14,c14/c15,c15/c16}
\draw[] (\from) to (\to);
\foreach \from/\to in 
{c1/c3,c2/c4,c3/c5,c4/c6,c5/c7,c6/c8,c7/c9,c8/c10,c9/c11,c10/c12,c11/c13,c12/c14,c13/c15,c14/c16}
\draw[bend left] (\from) to (\to);
\foreach \from/\to in 
{c1/c4,c2/c5,c3/c6,c4/c7,c5/c8,c6/c9,c7/c10,c8/c11,c9/c12,c10/c13,c11/c14,c12/c15,c13/c16}
\draw[bend right] (\from) to (\to);
 \draw[dotted] (ya3) -- (c1);
 \draw[dotted, ->] (c16) -- (ya4);
 \draw[bend left, <->] (0.5,2.5) to (6.5,0.7);
 \node at (3,2) {$\cong$};
\end{tikzpicture}
\caption{Illustration of the proof of Proposition \ref{pro: LWconv}: Any ball of radius $r < \lfloor n/2 \rfloor$ in graph-distance (indicated by the dotted rectangle) around any vertex in $G_{k,n}$ is isomorphic to the respective ball around $0$ in $G_{k,\infty}$. On the other hand, for given $r \geq 1$ we can choose $n \geq \lceil 2r \rceil$ so that each ball of radius $r$ in $G_{k,n}$ is isomorphic to the respective ball around $0$ in $G_{k,\infty}$.}
\label{fig: isomorphic}
\end{figure}
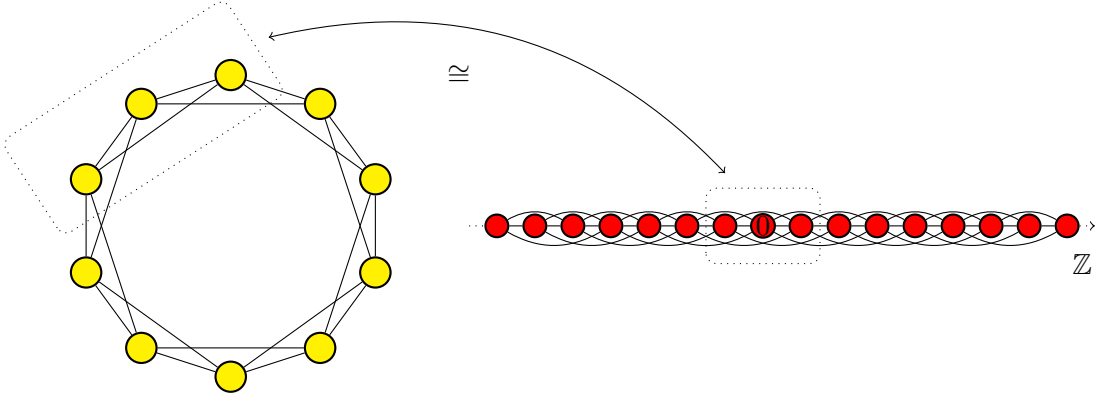
\end{proof}

From Propositions \ref{pro: disjoint regular} and \ref{pro: LWconv} we obtain the following theorem that produces a counter example to the power-law hypothesis:
\begin{thm}[PageRank of connected unions of circulant graphs]
\label{thm: Counterxample}
Let $p=(p_k)_{k \in \N}$ be any probability mass function such that $p_k=0$ for odd $k$. Then for the graphs $(\tilde{G}_n^{(p)})_{n \geq n_p}$, 
\begin{enumerate}[(a)]
    \item $\lim_{n \rightarrow \infty}\mathbb{P}(d_{\phi_n}=k)=p_k$.
  In particular, if $p$ is a power-law, then $d_{\phi_n}$ converges in distribution to a power-law with the same tail-index.
    \item $\boldsymbol{R}^{\sss(\tilde{G}_n^{(p)})}_{\phi_n} \stackrel{n \rightarrow \infty}{\longrightarrow} 1$ in probability.
\end{enumerate}
\end{thm}
\begin{proof}[Proof of Theorem \ref{thm: Counterxample}]
As $M_n=o(n)$, Statement (a) follows readily from Statement a) of Proposition \ref{pro: disjoint regular}. 
By Statement (b) of Proposition \ref{pro: disjoint regular},  we know that $\boldsymbol{R}^{\sss(G_n^{(p)}))}_{\phi_n} \stackrel{n \rightarrow \infty}{\longrightarrow} 1$ in probability and in distribution. By Proposition \ref{pro: LWconv}, the sequence $(\tilde{G}_n^{(p)}))_{n \in \N}$ of connected graphs converges locally weakly to the same random graph $(G_\infty,\phi)$ as the sequence $(G_n^{(p)}))_{n \in \N}$. Hence, employing Theorem \ref{thm: localConvPageRank}, we conclude that the sequences $(\boldsymbol{R}^{\sss(\tilde{G}_n^{(p)}))}_{\phi_n})$ and $(\boldsymbol{R}^{\sss(G_n^{(p)}))}_{\phi_n})$ have the same limit $R^{\sss (G_\infty)}_\phi$ in distribution. As the distributional limit of real-valued random variables is unique (e.g., \cite[Remark 13.13]{Kl20}) we conclude that $\boldsymbol{R}^{\sss(\tilde{G}_n^{(p)}))}_{\phi_n} \stackrel{n \rightarrow \infty}{\rightarrow} 1$ in distribution. In particular, convergence in distribution to a constant limit implies convergence in probability (e.g., \cite[Section 7.2 Thm. 4]{GrSt01}), which finishes the proof.
\end{proof}
\section{Discussion and open problems}
\label{sec: discussion}
In this section, we provide a discussion of our results and state some open problems. We start in Section \ref{sec-directed-settings} by discussing extensions of our work to directed graphs. In Section \ref{sec-diff-PAM}, we provide an intuitive explanation why the power-law hypothesis holds true for undirected preferential attachment models, while it fails for their directed versions. We close in Section \ref{sec-concl-outlook} by giving conclusions and open problems.

\subsection{Extensions to directed setting}
\label{sec-directed-settings}
In this section, we investigate to what extent our results extend to the directed setting. We discuss both a setting in which the PageRank can be upper bounded by the in-degree, as well as an extension of our counter example. We start by recalling how the PageRank vector is defined in the directed setting.

For a directed graph $G_n=(V_n, 
\vec{E_n})$ with $\#V_n=n$, in-degree vector $\boldsymbol{d}^-=(d_i^-)_{i \in V_n}$, and out-degree vector $\boldsymbol{d}^+=(d_i^+)_{i \in V_n},$ the directed analogue of \eqref{eq: PageRankMatrixUndirect} is defined as 
\begin{equation}
 p^{\sss(G_n)}_{ij}:=\frac{a_{ij}^{\sss(G_n)}}{d_i^+},    
\end{equation}
and the graph-normalized PageRank equation reads 
\begin{equation}\label{eq: PageRank matrixDir}
\boldsymbol{R}^{\sss(G_n)}=c\, \boldsymbol{R}^{\sss(G_n)} \boldsymbol{P}^{\sss(G_n)}+(1-c)\boldsymbol{1}_n, \quad \boldsymbol{R}^{\sss(G_n)} \in [0,n]^n.
\end{equation}
\subsubsection{Upper bound in directed setting with bounded ratio of in- and out-degrees}
The simple proof of Theorem \ref{thm: upperBoundUndirected} is based on employing the symmetry of the adjacency matrix, which is a stronger assumption than requiring the in- and out-degree vector to coincide. However, to extend the scope of Theorem \ref{thm: upperBoundUndirected} similar to \eqref{eq: PageRankSymmetry} and omitting the $G_n$-dependency in notation, we can write
\begin{equation}\label{eq: UpperBoundBoundedRatio1}
\begin{split}
\frac{R_i}{d_i^-}&=c\sum_{j \in [n]}\frac{a_{ji}}{d_i^-d_j^+}R_j+\frac{1-c}{d_i^-}=c\sum_{j \in [n]}\frac{d_j^-}{d_j^+}\frac{a_{ji}}{d_i^-}\frac{R_j}{d_j^-}+\frac{1-c}{d_i^-} \cr
&=c\sum_{j \in [n]}\frac{d_j^-}{d_j^+}p^\text{rev}_{ji}\frac{R_j^-}{d_j}+\frac{1-c}{d_i^-},
\end{split}
\end{equation}
where 
    \begin{equation}p^{\text{rev}}_{ij}:=\frac{a_{ji}}{d_i^-}
    \end{equation}
defines a stochastic matrix $\boldsymbol{P}^\text{rev}.$ 
Now assume that there is exists a $K < \infty$ such that       
    \begin{equation}
    \max_{i\in [n]} \frac{d_i^-}{d_i^+}\leq K.
    \end{equation}
Then defining $\boldsymbol{v}^{\text{dir}}$ by $v_i^\text{dir}:=R_i/d_i^-$ for $i\in[n]$, we obtain from \eqref{eq: UpperBoundBoundedRatio1}
\begin{equation}
\boldsymbol{v} \leq c \,K\boldsymbol{P}^\text{rev} \boldsymbol{v}+(1-c)\boldsymbol{Q}\boldsymbol{1}_n.    
\end{equation}
In case that $\Vert cK\boldsymbol{P}^\text{rev} \Vert <1$, which is certainly satisfied when $c<1/K$, we can invert the operator $\mathbbm{I}_n-cK\boldsymbol{P}$, to obtain
\begin{equation}\label{eq: UpperBoundBoundedRatio2}
(1-c)(\mathbbm{I}_n-c\boldsymbol{P}^\text{rev})^{-1}\boldsymbol{Q}\boldsymbol{1}_n \leq \boldsymbol{v} \leq (1-c)(\mathbbm{I}_n-cK\boldsymbol{P}^\text{rev})^{-1}\boldsymbol{Q}\boldsymbol{1}_n,    
\end{equation}
which is a generalization of \eqref{eq: UpperBoundAlmostThereWithV}.
Now, for the upper bound in \eqref{eq: UpperBoundBoundedRatio2}, we proceed as in the remainder of the proof of Theorem \ref{thm: upperBoundUndirected} to obtain the following result:
\begin{thm}[Upper bound on PageRank for bounded ratio of in- and out-degrees]\label{thm: UpperBoundBoundedRatio}
Let $G_n=(V_n,\vec{E}_n)$ be any any directed graph with $\# V_n=n$ such that  $m_n:=\min_{i \in V_n}d_i^{-} \geq 1$ and  $K_n:=\max_{i \in V_n}\frac{d_i^-}{d_i^+} < \frac{m_n}{c}$. Then
\begin{equation}
R_i^{\sss(G_n)} \leq  \frac{K_n}{m_n}d_i^-\quad \text{for all }i \in V_n.
\end{equation}
\end{thm}

A special case of Theorem \ref{thm: UpperBoundBoundedRatio} arises for directed {\em Eulerian} graphs, in which $d_i^+=d_i^-\geq 1$ for all $i\in V_n$, and for which $K_n=m_n=1.$ On the other hand, the directed version of the preferential attachment model clearly fails the assumption of a bounded ratio of out- and in-degrees, which aligns with the known result that the limiting root-PageRank for that model has heavier tails than the in-degree of the root (\cite[Theorem 3.1]{BaOC22}).
\begin{rk}[Interpretation of $\boldsymbol{P}^{\mathrm{rev}}$]
{\normalfont \rmfamily In view of the \textit{random surfer} interpretation (see \cite[Section 2.1.2]{BrPa98}) of the PageRank vector as the stationary distribution of a random walker who at each step with probability $c$ chooses one of the outgoing edges uniformly at random and with probability $1-c$ jumps to a vertex chosen uniformly at random from the entire vertex set $V_n$, the matrix $\boldsymbol{P}^{\mathrm{rev}}$ describes the reversed random walk, in which the random walker follows edges in the opposite direction. Such random walk, too, was employed for characterizing other vertex centrality measures, e.g.\ CheiRank centrality \cite{ZhZhSh10}.}\hfill\ensymboldefinition    
\end{rk}

\subsubsection{Lower bound in directed setting} \cite[Theorem 2.1]{GavdHLi20} relates local weak convergence of graph sequences to  convergence in distribution of the associated root-PageRank. This theorem covers the case of \textit{directed} graphs as well, with a specific form of local convergence (where we consider the in-components and keep the out-degrees of vertices as vertex marks). Taking that into account, an inspection of the proof of Proposition \ref{pro: PageRankLowerBound} immediately leads to the following generalization, where we now introduce
    \begin{equation*}
        \dphigeqalphadir=\#\{j \mid j \rightarrow \phi \text{ and } d^+_j \geq \alpha\}.
    \end{equation*}
\begin{pro}[Lower bound on PageRank for directed graphs]
\label{pro: LowerBoundDirected}
Let $G=(V,\vec{E},\phi)$ be an infinite rooted directed random graph with in-degree vector $\boldsymbol{d}^-$ and out-degree vector $\boldsymbol{d}^+$ which arises as the local weak limit of a sequence $(G_n)_{n \in \N}$ of directed graphs. Let further $R_\phi$ denote the limiting (in distribution) root-PageRank. Assume that there are $\alpha>0$ and $\varepsilon>0$ such that, as $k\rightarrow \infty,$
\begin{equation}\label{eq: LowerBoundConditionDirected}
\mathbb{P}\left(d^{-}_\phi >k, \dphigeqalphadir\geq (1-\varepsilon)d^-_\phi\right)=o(\mathbb{P}(d^-_\phi >k)).
\end{equation}
Then,
\begin{equation}
\mathbb{P}(R_\phi>k) \geq (1+o(1)) \mathbb{P}\Big(d^-_\phi>\frac{\alpha}{\varepsilon c(1-c)}k\Big) \quad \text{as }k \rightarrow \infty.   
\end{equation}
\end{pro}

Also Corollary \ref{cor: lowerbound} can be obviously generalized to the directed setting, but we refrain from doing so.

As a particular application of Proposition \ref{pro: LowerBoundDirected}, we directly obtain the bound 
\[\mathbb{P}(R_\phi>k) \geq (1+o(1))\mathbb{P}(d_\phi^{-}>\beta k) \quad \text{as }k \rightarrow \infty
\]
for the limiting root-PageRank of the directed version of the preferential attachment model, for any $\beta>1/[c(1-c)]$. To see this, note that in the directed version of the P\'olya point tree (\cite[Definition 4.4]{BaOC22}), each vertex has out-degree $1$, so \eqref{eq: LowerBoundConditionDirected} is fulfilled for any choices of $\alpha>1$ and $0<\varepsilon<1$.
However, this lower bound is \textit{not} sharp, as the tails of the limiting root-PageRank are actually heavier than those of the in-degree (\cite[Theorem 3.1]{BaOC22})
\subsection{Intuition for the difference in directed and undirected preferential attachment models}
\label{sec-diff-PAM}
It is surprising that the tail of the PageRank is drastically different in the directed and undirected preferential attachment models. Indeed, in the directed model, PageRank has a heavier tail than the degrees \cite{BaOC22,}. Furthermore, the authors in \cite{AnBaBhPi23} prove that this result persists in multi-type preferential attachment networks and even in multi-type uniform attachment networks. In the latter case, surprisingly, the degree-distribution of the vertices of each type is geometric, while PageRank, for all types, has a power law distribution with the same exponent for each type of vertices \cite[Theorem 5.3(b)]{AnBaBhPi23}. 
At the same time, our Theorem~\ref{thm: upperBoundUndirected} states that in any undirected graph, PageRank cannot have a heavier tail than the degree. An intuitive explanation of this drastic difference between directed and undirected graphs stems from the fact that PageRank of a vertex is proportional to the sum of the probabilities of all paths leading to this vertex \cite{AvLi06}; consequently, the limiting distribution of PageRank is defined by the sum of the  probabilities of all paths leading to the root vertex in the local weak limit. Such path-counting arguments have been successfully applied also, e.g., to the personalized PageRank with restarts on undirected graphs \cite{AvVdHSo14}.
In the directed preferential attachment model, the local weak limit is a directed tree, and the sum of the probabilities of all (directed) paths to the root equals the discounted tree size. Specifically, the size of the generation at distance $l$ from the root, is discounted by the factor $(c/m)^l$ with $m$ being the constant and deterministic out-degree. Furthermore, this tree is a realization of a continuous-time branching process, stopped at a random time (more precisely, this random time has an exponential distribution with parameter equal to the growth rate of the process, i.e., the so-called Malthusian parameter). Up to that random time, the degree of the root grows exponentially in time, but the discounted tree size grows exponentially as well, at a faster rate. These growth rates enter the power-law exponent making the tail of PageRank heavier than that of the degree.  

In contrast, in the undirected preferential attachment model, a step on the path through vertex $v$ is discounted by the factor $c/d_v$. Our analysis shows that this difference is crucial:  it makes PageRank of the root grow at most as fast as its degree. Note also that in the undirected model, the sum of the probabilities of all paths to the root is not the same as the discounted tree size because edges can be traversed both ways. This makes the path counting in the undirected case more difficult, but this is not the feature that affects the power law exponent of PageRank because each extra step is penalized by the factor $c$. 

Altogether, we believe that the differences between the directed and undirected models are educational, and their better understanding will lead to new results in the future. 
\subsection{Conclusion and outlook}
\label{sec-concl-outlook}
The upper bound for PageRank stated in Theorem \ref{thm: upperBoundUndirected} holds true for every finite undirected graph. It also extends to the case of directed graphs with a bounded ratio of degrees. On the contrary, the counter-example presented in Theorem \ref{thm: Counterxample} shows that a similar lower bound fails when the degrees of vertices are equal to those of their neighbors (i.e., maximal degree-degree correlations).
Proposition \ref{pro: PageRankLowerBound} provides a sufficient assumption for an asymptotic lower bound, which is particularly easy to check for sequences of graphs which converge to (branching-process) trees in the local weak sense.
Combining Theorem \ref{thm: upperBoundUndirected} and Proposition \ref{pro: PageRankLowerBound}, we obtain a statement on the asymptotic behavior of PageRank which holds under assumptions that are reasonably simple to check in applications, which is one of the strengths of the result.

Compared to the related work in e.g., \cite{JeOC10,ChLiOl17,OC21}, the proofs in the present paper are not based on studying a stochastic fixed-point equation on a suitable (marked) branching-process tree, but instead rely on applying the recent result of \cite{GavdHLi20} on convergence in distribution for the root-PageRank associated with locally weakly converging sequences of graphs. This simplification, however, comes at the expense of a somewhat weaker result, in that we cannot identify the asymptotic constant $a$ such that $\mathbb{P}(R_\phi>k)=a\mathbb{P}(d_\phi>k)(1+o(1))$ as $k\rightarrow \infty.$ 
It would be of interest to identify such a constant in general like it is the case, e.g., for the directed configuration model (\cite[Theorems 6.4 \& 6.6]{ChLiOl17}).

It remains an open question whether the sufficient Condition \eqref{eq: LowerBoundMainAssumtion} in Proposition \ref{pro: PageRankLowerBound} for the respective asymptotic lower bound actually is an if and only if condition. Moreover, we believe Condition \eqref{eq: LowerBoundMainAssumtion} to hold much more generally than for the tree-scenarios considered in the applications, with the counterexample presented in Theorem \ref{thm: Counterxample} being a rather artificial exception. It would hence be of interest to further examine the scope of our main result.

\paragraph{Acknowledgements.}
The work of RvdH and NL is supported in part by the Netherlands Organisation for Scientific Research (NWO) through Gravitation-grant {\normalfont \sffamily NETWORKS}-024.002.003.

\printbibliography
\end{document}